\documentclass[12pt,english]{article}
\usepackage[T1]{fontenc}
\usepackage[utf8]{inputenc}
\usepackage{geometry}
\usepackage[english]{babel} 
\usepackage{float}
\usepackage{amsmath}
\usepackage{amsthm}
\usepackage{amssymb}
\usepackage{graphicx}
\usepackage[unicode=true,pdfusetitle,
 bookmarks=true,bookmarksnumbered=false,bookmarksopen=false,
 breaklinks=false,pdfborder={0 0 1},backref=false,colorlinks=false]
 {hyperref}

\usepackage{booktabs} 
\usepackage{float}    

 \usepackage{algorithm}
 \usepackage{enumitem}
\usepackage{algorithmic}
\usepackage{xcolor}

\usepackage{subcaption} 

\makeatletter
\usepackage{comment}
\DeclareMathOperator{\ii}{Im}

\theoremstyle{plain}
\newtheorem{thm}{\protect\theoremname}

\theoremstyle{plain}
\newtheorem{rem}[thm]{\protect\remarkname}
\theoremstyle{plain}
\newtheorem{prop}[thm]{\protect\propositionname}
\graphicspath{{./Figures/}}
\newcounter{example}
\newenvironment{example}[1][]{\refstepcounter{example}\par\medskip
   \noindent \textbf{Example~\theexample. #1} \rmfamily}{\medskip}

\usepackage{diagbox}

\makeatother

\providecommand{\propositionname}{Proposition}
\providecommand{\remarkname}{Remark}
\providecommand{\theoremname}{Theorem}
\date{}

\title{Neumann series of Bessel functions in direct and inverse spherically symmetric transmission eigenvalue problems}

\author{%
  Vladislav V. Kravchenko\textsuperscript{1}\thanks{vkravchenko@math.cinvestav.mx}%
  \and
  L. Estefania Murcia-Lozano\textsuperscript{2}\thanks{mursiia@sfedu.ru}%
  \and
  Nikolaos Pallikarakis\textsuperscript{3}\thanks{npall@central.ntua.gr}%
}

\date{%
  \small
  \textsuperscript{1}Departamento de Matemáticas, Cinvestav, Unidad Querétaro,\\
    Libramiento Norponiente 2000, Fracc. Real de Juriquilla, Querétaro, Qro., 76230 México\\[1ex]
  \textsuperscript{2}Regional Mathematical Center, Southern Federal University, Rostov-on-Don 344090, Russia\\[1ex]
  \textsuperscript{3}Department of Mathematics, National Technical University of Athens,\\
    Zografou Campus, Athens, 15780, Greece\\[2ex]
}

\begin{document}
\maketitle
\let\thefootnote\relax\footnotetext{Corresponding author: \texttt{npall@central.ntua.gr}}

\begin{abstract}
The transmission eigenvalue problem (TEP) plays a central role in inverse scattering theory. Despite substantial theoretical progress, the numerical solution of direct and inverse TEP in spherically symmetric domains with variable refractive index— covering real and complex eigenvalues—remains challenging. This study introduces a novel Neumann Series of Bessel Functions (NSBF) methodology to address this challenge. After reformulating the TEP as a Sturm–Liouville equation via a Liouville transformation, we expand its characteristic function in an NSBF whose coefficients are computed by simple recursive integration. In the direct problem, eigenvalues—real or complex—are found by root‑finding on a truncated NSBF partial sum, yielding high accuracy with a few coefficients, as demonstrated with various examples.  For the inverse problem, we develop a two-step approach: first, recovering the transformed interval length $\delta$ from spectral data via a new NSBF-based algorithm, and second, reconstructing the refractive index $n(r)$ by solving a linear system for the first NSBF coefficients. A spectrum completion technique is also implemented to complete the spectrum and solve the corresponding inverse problem when eigenvalue data is limited. Numerical examples confirm the method’s robustness and accuracy across a wide range of refractive indices, with no a priori assumptions on $\delta$ or the sign of the contrast $1-n(r)$.

\end{abstract}

\textbf{keywords:} transmission eigenvalues, spherically symmetric domain, direct spectral problem, inverse spectral problem, Neumann series of Bessel functions, spectrum completion

\section{Introduction}
\label{sec_intro} 

The transmission eigenvalue problem (TEP) is a fundamental non-self-adjoint boundary value problem that arises in the context of inverse scattering theory for inhomogeneous acoustic, electromagnetic, and elastic media. First introduced by Colton, Monk, and Kirsch \cite{CM,Ki}, this problem has been extensively studied, both theoretically and computationally, due to its applications in reconstructing the material properties and support of a medium and its relation to non-scattering wave phenomena. We refer to the following books, monographs and surveys for a detailed review of the subject \cite{CCqual, CCH, CH, CKbook, KBook, LiuSur, Pallreview}. 

More specifically, the TEP is defined as a boundary value problem involving a coupled set of equations with accompanying transmission boundary conditions. In the case of acoustic scattering for an isotropic and inhomogeneous medium, the interior TEP is given by the system
\begin{eqnarray}
&\Delta w+k^2n(x)w=0&\ \textrm{in}\ D, \label{trD1}\\
&\Delta v+k^2v=0&\ \textrm{in}\ D, \label{trD2}\\
&w=v&\ \textrm{on}\ \partial D, \label{trD3}\\
&\frac{\partial w}{\partial \nu}=\frac{\partial v}{\partial \nu}&\ \textrm{on}\ \partial D.  \label{trD4}
\end{eqnarray} 
The domain $D\subset \mathbb{R}^n$ is assumed to be simply connected with a Lipschitz boundary $\partial D$ and $\nu$ is the outward unit normal. The refractive index $n(x):=c_0^2/c^2(x)$ is defined as the ratio of the square of the reference sound speed to the square of the local sound speed in the medium, and is assumed to be equal to one outside the inhomogeneous region. The complex values of $k$ corresponding to non-trivial solutions of (\ref{trD1})-(\ref{trD4}) are the transmission eigenvalues and $(w,v)$ the eigenfunctions.  The direct TEP consists of finding $\{k;(w,v)\}$ for given $n(x)$ while the inverse TEP is to recover the unknown refractive index from the knowledge of the spectrum.

Research on the interior transmission problem has primarily focused on the discreteness of the spectrum, which is crucial for sampling methods used in reconstructing the support of inhomogeneous media \cite{CKP89,CKi}. In \cite{McP}, it was first shown that transmission eigenvalues provide information about the refractive index, a result later extended to material characterization and non-destructive testing \cite{ccc, harris2014inverse}. The fact that real transmission eigenvalues can be measured from scattering data \cite{CCHdeter} has made their use particularly important. However, the existence of an infinite, discrete spectrum remained unresolved for many years until \cite{CGH}, due to the lack of a standard theory for non-self-adjoint eigenvalue problems. 

The special case of the TEP for spherically symmetric domains has been of particular interest in the research community. The use of spherical coordinates simplifies the study, enabling the application of analytical methods, integral equations, and tools from complex analysis to gain a deeper understanding of the subject.  For a review of the state-of-the-art results we refer to \cite{Pallreview} and the references therein, for the main findings from the late 80s to the present.  

Significant research has been devoted to numerically solving the direct eigenvalue problem for general domains, see e.g. \cite[Section 6]{SZ}, which is an actively evolving subject. In contrast, the numerical solution of the inverse eigenvalue problem for general refractive indices is less studied, primarily because of the inherent complexity of the problem \cite{GP1,GH}. When restricted to the spherically symmetric problem, computational results have been presented in various studies. In domains with cylindrical or spherical symmetry, transmission eigenvalues can be computed analytically using separation of variables for constant or piecewise constant refractive indices, as e.g. in examples presented in \cite{CMSn, Kn}. In some cases of variable refractive indices, eigenvalues can also be derived analytically \cite{CL2,CLM}. Additionally, reconstruction algorithms for the inverse spherically symmetric problem are included in \cite{AP,MPS,MSS,WZS,WW20,XY,XYB}, with some of them also presenting numerical examples for specific cases. Despite these advancements, a general numerical method for the direct and inverse problem with variable spherically symmetric refractive indices (covering both real and complex eigenvalues) remains open. The present study aims to fill this gap by applying a novel Neumann Series of Bessel Functions (NSBF) methodology. 

NSBF representations for solutions of Sturm-Liouville equations were obtained in \cite{KNT2017AMC, KT2018Calcolo} as a corollary of Fourier-Legendre series expansions of transmutation (transformation) operator integral kernels. The use of NSBF for solving direct and inverse spectral problems (see, e.g., \cite{KravSC, KKC, CKK2024MMAS, Kr2024LJM, Kr2025JIIP,KravchenkoMurcia2025}) is due to their several important features: uniform convergence with respect to the spectral parameter in any strip of the complex plane parallel to the real axis (see Theorem \ref{Th NSBF old} below); the first coefficient of the series is sufficient for recovering the Sturm-Liouville equation; simple and efficient criteria for controlling accuracy of the approximation by partial sums.  

In this work, NSBF representations simplify solving the direct TEP by reducing it to computing NSBF coefficients using a recurrent integration procedure. The solution then involves locating the zeros of the resulting NSBF partial sum. The solution of the inverse TEP consists of two steps: the recovery of the values of the NSBF coefficients at the endpoint of the interval from the transmission eigenvalues, and the recovery of the refractive index from the first NSBF coefficients obtained by solving a system of linear algebraic equations. Since the inverse problem involves the Liouville transformation, which relates the Sturm-Liouville equation in string form to the Schrödinger equation, we must address an important issue that arises. This concerns computing the unknown length of the interval after the Liouville change of variable, which will be denoted as $\delta$ throughout this text. This problem is not specific for TEP and, on the contrary, naturally arises in different applications, whenever the Liouville transformation is involved \cite{Erd, Glad}. One of the contributions of the present study is a new approach for recovering this $\delta$, based on the properties of the NSBF representations. Additionally, we explore the possibility of the spectrum completion, that is to compute higher-order transmission eigenvalues, from an initial relatively small set of real and/or complex eigenvalues. This develops the idea of \cite{KravSC}, adapted to the TEP. These considerations lead to the main objective of this paper, which is to propose a new method for the approximate solution of direct and inverse transmission eigenvalue problems in the spherically symmetric case.

The structure of this manuscript is organized as follows. In Section \ref{sec_nsbf}, we introduce the key theoretical concepts of the NSBF method, with a particular focus on Sturm-Liouville-type problems. Section \ref{sec_ss_tep} presents the main definitions and formulations for the spherically symmetric TEP, including a new Liouville transformation and the relevant characteristic functions. In Section \ref{sec_comp_nsbf}, we implement our NSBF computational method to solve the TEP.  This section includes algorithms for the direct problem in Subsection \ref{sec_comp_nsbf_dir} and for the inverse problem in Subsection \ref{sec_comp_nsbf_inv}. Additionally, we develop a method to recover the transformed interval length $\delta$  from spectral data and we also introduce a spectrum completion methodology, both leveraging the NSBF representations. Section \ref{sec_exampl} provides several examples that demonstrate the applicability and effectiveness of our methods, with specific examples for the direct problem in Subsection \ref{sec_exampl_dir} and the inverse problem in Subsection \ref{sec_exampl_inv}. Finally, we conclude with a discussion and summary in Section \ref{sec_fin}.

\section{Neumann series of Bessel functions representations for Sturm-Liouville problems} \label{sec_nsbf}

In this section, for the reader's convenience and to maintain the manuscript's conciseness, we present only the key Neumann Series of Bessel Functions results relevant to our study. For a more detailed review of the subject, we refer to the comprehensive bibliography, such as \cite{krBook2020}.

Let $q\in\mathcal{L}_{2}(0,L)$ be a complex-valued potential and $L>0$. Consider
the Sturm-Liouville equation
\begin{equation}
-y''+q(x)y=\rho^{2}y,\,\,0<x<L,\label{eq:PrincipalEq}
\end{equation}
where $\rho\in\mathbb{C}$ is the spectral parameter. By $S\left(\rho,x\right),\ \phi(\rho,x)$  and $T\left(\rho,x\right)$ we denote the solutions
of (\ref{eq:PrincipalEq}) satisfying the initial conditions 
\begin{align*}
S(\rho,0) & =0,\,S'(\rho,0)=1, \\
T(\rho,L) & =0,\,T'(\rho,L)=1,\\
\phi(\rho,0) & =1,\,\phi'(\rho,0)=0.\label{eq:h}
\end{align*}
These solutions satisfy the identity 
\begin{equation}
T(\rho,x)=\phi(\rho,L)S(\rho,x)-\phi(\rho,x)S(\rho,L).\label{eq:Identity}
\end{equation}

\begin{thm} (\cite[Theorem 4.1]{KNT2017AMC})
\label{Th NSBF old}
Let $q\in\mathcal{L}_{2}(0,L)$. The solutions $S\left(\rho,x\right)$
and $\phi(\rho,x)$ admit the following series representation
\begin{align*}
S(\rho,x) & =\frac{\sin(\rho x)}{\rho}+\frac{1}{\rho}\sum_{n=0}^{\infty} s_{n}(x)j_{2n+1}(\rho x),\\
\phi(\rho,x) & =\cos(\rho x)+\sum_{n=0}^{\infty} g_{n}(x)j_{2n}(\rho x),
\end{align*}
where $j_{n}(z)$ stands for the spherical Bessel function of order
$n$, see, e.g., \cite{abromowitz1972handbook}.

The series converge pointwise with respect to $x$ for $x\in[0,L]$. Additionally, for
every $x\in[0,L]$ the series converge uniformly in any strip of the complex
plane of the variable $\rho$, parallel to the real axis. In particular the remainders of their partial sums
\begin{align}
S_{N}(\rho,x) & =\frac{\sin(\rho x)}{\rho}+\frac{1}{\rho}\sum_{n=0}^{N-1}s_{n}(x)j_{2n+1}(\rho x), \label{eq:STrn} \\
\phi_{N}(\rho,x) & =\cos(\rho x)+\sum_{n=0}^{N-1}g_{n}(x)j_{2n}(\rho x). \label{eq:PhiTrun} 
\end{align}
admit the estimates
\[
\left|\rho S(\rho,x)-\rho S_{N}(\rho,x)\right|\leq\frac{\tilde{\varepsilon}_{N}(x)\sinh\left(Cx\right)}{C}\,\,\text{and}\,\,\text{\ensuremath{\left|\phi(\rho,x)-\phi_{N}(\rho,x)\right|\leq\frac{\tilde{\varepsilon}_{N}(x)\sinh\left(Cx\right)}{C}},}
\]
for all $\rho$ belonging
to the strip $|\ii\rho|\leq C$, $C> 0$, where $\tilde{\varepsilon}_{N}(x)$ is a positive function tending to zero when $N\rightarrow\infty$. 
\end{thm}

The coefficients $g_{n}(x)$ and $s_{n}(x)$ can be calculated
following a simple recurrent integration procedure starting with
\begin{equation}\label{eq:g0s0}
	g_{0}(x)=\phi(0,x)-1, \quad s_{0}(x)=3\left(\frac{S(0,x)}{x}-1\right),
\end{equation}
see Remark \ref{recursive} below.

\begin{rem} (\cite{KNT2017AMC}) \label{recursive}
Let $f$ be a solution of the equation $\displaystyle f''-q(x)f=0$ on the interval $(0,L)$ such that $f(0)=1$ and $f^{\prime}(0)=0$.   
    Consider the functions
        $$\sigma_{2n}(x):=x^{2n}\frac{g_n(x)}{2} \quad \text{and} \quad \sigma_{2n+1}(x):=x^{2n+1}\frac{s_n(x)}{2}, \quad n=0,1,\ldots.$$
The coefficients $s_n(x)$ and $g_n(x)$ are obtained with the aid of the recurrent formulas for the functions $\sigma_n$:
$$\sigma_{-1}(x):=\frac{1}{2x} \quad \sigma_{0}(x):=\frac{f(x)-1}{2},$$
$$\eta_n(x):=\int_0^x (tf'(t)+(n-1)f(t))\sigma_{n-2}(t)dt, \quad \theta_n(x)=\int_0^x\frac{1}{f^2(t)}(\eta_n(t)-tf(t)\sigma_{n-2}(t))dt,$$
$$\sigma_n(x)=\frac{2n+1}{2n-3}\left( x^2 \sigma_{n-2}(x) +c_nf(x)\theta_n(x)\right),$$
for $n=1,2,\ldots,$ where $c_n=1$ if $n=1$ and $c_n=2(2n-1)$ otherwise. We refer to \cite{krBook2020} and references therein. 
\end{rem}

Analogously, the solution $T(\rho,x)$ admits the series representation
\begin{equation}
T(\rho,x) =\frac{\text{sin}(\rho(x-L))}{\rho}+\frac{1}{\rho}\sum_{n=0}^{\infty}t_{n}(x)j_{2n+1}(\rho(L-x)),
\end{equation}
and denote its partial sum by 
\begin{equation}
T_N(\rho,x) =\frac{\text{sin}(\rho(x-L))}{\rho}+\frac{1}{\rho}\sum_{n=0}^{N-1}t_{n}(x)j_{2n+1}(\rho(L-x)).\label{eq:TTrun}
\end{equation}

\begin{rem} (\cite{KNT2017AMC})
The coefficients $s_n$ and $g_n$ satisfy: 
\begin{equation}
\omega(x) = \sum_{n=0}^{\infty} \frac{g_n(x)}{x} = \sum_{n=0}^{\infty} \frac{s_n(x)}{x}, \label{Eqsn}    
\end{equation}
where $\omega(x)=\frac{1}{2}\int_0^x q(t)dt$.
\end{rem}
\begin{rem}\label{RemarkAbelForS}
In the case $\rho=0$, given the solution $\phi(0,x)$, the second linearly independent solution $S(0,x)$ is provided by the Abel formula 
\[
S(0,x)=\phi(0,x)\int_{0}^{x}\frac{1}{\phi^2(0,t)}dt. 
\]
Hence,
the following relation between the coefficients $g_{0}(x)$ and $s_{0}(x)$
holds
\[
\left(\frac{s_{0}(x)}{3}+1 \right)x=\left(g_{0}(x)+1\right)\int_{0}^{x}\frac{1}{\left(g_{0}(t)+1\right)^{2}}dt.
\]
\end{rem}

\section{The spherically symmetric transmission eigenvalue problem} \label{sec_ss_tep}

We consider the interior TEP defined in the unit ball of $\mathbb{R}^3$,  $B:=\{ x\in\mathbb{R}^3: |x|<1 \}$, for a real-valued refractive index $n(|x|):=n(r)$ which is a function depending only on the radial variable. Problem (\ref{trD1})-(\ref{trD4}) is then rewritten as 
\begin{eqnarray}
&\Delta w+k^2n(r)w=0&\ \textrm{in}\ B, \label{trB1}\\
&\Delta v+k^2v=0&\ \textrm{in}\ B, \label{trB2}\\
&w=v&\ \textrm{on}\ \partial B, \label{trB3}\\
&\frac{\partial w}{\partial r}=\frac{\partial v}{\partial r}&\ \textrm{on}\ \partial B. \label{trB4}
\end{eqnarray} 
The spherically symmetric transmission eigenvalues correspond to the complex values of $k$ for which non-trivial solutions exist to the system of equations (\ref{trB1})-(\ref{trB4}).

By introducing the spherical coordinates $(r,\theta,\phi)$ and applying separation of variables, (\ref{trB1})-(\ref{trB4}) simplifies to a  boundary value problem where the spectral parameter $k$ appears in the boundary condition at the right end-point, see e.g., \cite[Section 9.4]{CKbook} and \cite{CM,CP}.
More specifically, if we restrict to the spherically symmetric TEP when the eigenfunctions
are also axially symmetric, we are led to the following eigenvalue
problem
\begin{equation}
y''(r)+k^{2}n(r)y(r)=0,\quad0<r<1,\label{eq:MainSL}
\end{equation}
\begin{equation}
y(0)=0, \label{eq:dirichletCond} 
\end{equation}
with characteristic function
\begin{equation}
D_0(k):=\frac{\sin k}{k}y'(1)-\cos(k)y(1), \label{eq:CharEq}
\end{equation}
together with the normalization condition $y'(0)=1$. The eigenvalues of the problem \eqref{eq:MainSL}-\eqref{eq:CharEq} are called special transmission eigenvalues and are the zeros of $D_0(k)$. We also note that the entire function $D_0(k)$ vanishes at $k=0$, \cite[Theorem 2.4]{AGP}. Furthermore, since $n(r)$ is real-valued, if $k$ is an eigenvalue, then its complex conjugate $k^{*}$ is also an eigenvalue.
The direct problem involves determining $\{k;y\}$ for a given refractive index, while the inverse problem aims to reconstruct the unknown $n(r)$ from the known spectrum.

We assume that $n(r)$ is a function in $C^{1}[0,1]$ with $n''\in\mathcal{L}_{2}(0,1)$. It is common in the literature to suppose that the refractive index is sufficiently smooth in the boundary, i.e., $n(1)=1$ and $n^{\prime}(1)=0$. This is a natural assumption, aligning with the relevant  acoustic scattering problem \cite{CKbook,KBook}. For the needs of our study, this restriction is not necessarily imposed. 

We define the \textquotedblleft less-conventional\textquotedblright\  Liouville transformation, specifically tailored to the framework of the present work
\begin{equation}
\zeta(r):=\int_{r}^{1}\sqrt{n(t)}dt,\label{SLTransform}\end{equation}
\begin{equation}\label{ChangeVariableSolutions}
y(r)=z(\zeta)n^{-1/4}(r),\quad r=r(\zeta)
\end{equation}
and the quantity $\delta:=\zeta(0)$. This parameter is physically understood as the travel time for a wave to propagate from $r=0$ to $r=1$, see \cite{AGP}.

\begin{rem}
The Sturm-Liouville equation \eqref{eq:MainSL} can be transformed into a Schrödinger equation via a Liouville transformation defined as    \begin{equation}
\xi(r):=\int_{0}^{r}\sqrt{n(t)}dt, \label{Ltrans}
\end{equation}
see e.g. \cite{CM, McP}. Although this transformation is commonly used for solving transmission eigenvalue problems, it results convenient for us to consider the alternative transformation \eqref{SLTransform}. This allows us to deal with the solutions $\phi(\rho,x)$ and $S(\rho,x)$ satisfying initial conditions at the origin. 
\end{rem}

Using (\ref{SLTransform})-(\ref{ChangeVariableSolutions}), we can transform \eqref{eq:MainSL}-\eqref{eq:CharEq} into a canonical Sturm-Liouville problem in terms of the function $z(\zeta)$

\begin{equation}
-\ddot{z}(\zeta)+p(\zeta)z(\zeta)=k^{2}z(\zeta),\quad 0<\zeta<\delta,\label{eq:SL}
\end{equation}
\begin{equation}
z(k,\delta)=0,\quad \dot{z}(k,\delta)=-n^{-1/4}(0),\label{eq:conditionsTransformed}
\end{equation}
where $\dot{z}$ denotes the derivative with respect to $\zeta$. The potential $p(\zeta)\in\mathcal{L}_{2}(0,\delta)$ is given
by 
\[
p(\zeta(r)):=\frac{n''(r)}{4(n(r))^{2}}-\frac{5(n'(r))^{2}}{16(n(r))^{3}}.
\]

By substituting in \eqref{eq:CharEq} the expressions for $y(1)$ and $y'(1)$ obtained from \eqref{ChangeVariableSolutions}, $D_0(k)$ can be rewritten as
\begin{equation}
D_0(k)=\left(\frac{\cos k}{n^{1/4}(1)}+\frac{n'(1)\sin k}{4n^{5/4}(1)\,k}\right)z(k,0)+n^{1/4}(1)\frac{\sin k}{k}\dot{z}(k,0).\label{eq:CharEq2}
\end{equation}
We note that (\ref{eq:CharEq2}) is defined at the left endpoint $\zeta=0$. The equivalent characteristic function (for $\xi=\delta$) using the transformation (\ref{Ltrans}) is given in \cite[Eq. (3.10)]{CLM}.

Furthermore, \eqref{eq:CharEq} can be expressed in terms of functions $\phi(k,\delta)$ and $S(k,\delta)$
as follows.
\begin{prop}\label{CharEq}
The characteristic function $D_0(k)$ is equivalent to
 \begin{equation}
D_0(k)=a(k)\phi(k,\delta)+b(k)S(k,\delta),\quad k\in\mathbb{C},\label{eq:NewCharEq}
\end{equation}
where $\phi(k,\zeta)$ and $S(k,\zeta)$ are fundamental solutions of the Sturm-Liouville equation  \eqref{eq:SL} and \begin{equation*}
a(k) :=n^{1/4}(1)\frac{\sin k}{k},\quad b(k) :=-\left(\frac{\cos k}{n^{1/4}(1)}+\frac{n'(1)\sin k}{4n^{5/4}(1)\,k}\right).
\end{equation*}
\end{prop}

\begin{proof}
Solution $z(k,\zeta)$ satisfying (\ref{eq:conditionsTransformed})
is expressed in terms of the fundamental system of solutions $\left\{ \phi(k,\zeta),S(k,\zeta)\right\} $
as
\[
z(k,\zeta)=-n^{-1/4}(0)\left(S(k,\zeta)\phi(k,\delta)-\phi(k,\zeta)S(k,\delta)\right),\,\,k\in\mathbb{C},
\]
which implies
\begin{equation}
z(k,0)=n^{-1/4}(0)S(k,\delta)\ \ \textrm{and}\ \ z'(k,0)=-n^{-1/4}(0)\phi(k,\delta).\label{eq:ConditionsEquivalence}
\end{equation}
Substitution of (\ref{eq:ConditionsEquivalence}) in (\ref {eq:CharEq2})
gives us (\ref{eq:NewCharEq}).
\end{proof}

 \section{NSBF computational method for the transmission eigenvalue problem} \label{sec_comp_nsbf}

In the present section, we aim to apply the NSBF methodology to address both the direct and inverse spherically symmetric transmission eigenvalue problems. 
 
\subsection{Direct transmission eigenvalue problem} \label{sec_comp_nsbf_dir}

We solve the direct TEP, i.e. compute the real and complex eigenvalues of \eqref{eq:MainSL}-\eqref{eq:CharEq}, for a known refractive index $n(r)$. To do so, we need to find the zeros $k$ of the
expression (\ref{eq:NewCharEq}) for the characteristic function $D_0(k)$. For this, we consider the approximation  $D_{0,N}(k)$ of the characteristic function $D_{0}(k)$ (see Proposition \ref{CharEq}) by the truncated NSBF representations (\ref{eq:STrn}) and (\ref{eq:PhiTrun}), that is,

\begin{equation}
D_{0,N}(k)=a(k)\phi_N(k,\delta)+b(k)S_N(k,\delta),\quad k\in\mathbb{C}.\label{eq:NewCharEqTrunca}
\end{equation}
 The steps for solving the direct problem are presented in Algorithm \ref{algo_dir}.
 
\begin{algorithm}
\caption{The direct transmission eigenvalue problem.}
\label{algo_dir}
\begin{minipage}[b]{1\textwidth}
Assume the refractive index $n(r)$ is given.
\begin{enumerate}[itemsep=0.1ex, topsep=0ex]
\item Apply the Liouville transformation (\ref{SLTransform}) to obtain $p(\zeta(r))$ and $\delta$.

\item Calculate the approximate solutions (\ref{eq:STrn}) and (\ref{eq:PhiTrun}) for the corresponding Sturm-Liouville equation, where an optimal $N$ is obtained using Remark \ref{RemarkEpsilon}. 

\item  Approximate the characteristic function by \eqref{eq:NewCharEqTrunca}, that is
\begin{align}
 &D_{0,N}(k)=a(k)\text{cos}(k\delta)+a(k)\sum_{n=0}^{N-1}g_{n}(\delta)j_{2n}(k\delta)
 +b(k)\frac{\text{sin}(k\delta)}{k}+\frac{b(k)}{k}\sum_{n=0}^{N-1}s_{n}(\delta)j_{2n+1}(k\delta). \label{eq:NSBFDirectEqs-1}
\end{align}
\item  Compute the transmission eigenvalues by locating the zeros of $D_{0,N}(k)$, using Remarks \ref{RemarkPrincipleAlg} and \ref{RemarkSplines}. 
\end{enumerate}
\end{minipage}
\end{algorithm}

\begin{rem} \label{RemarkEpsilon}
The choice of an appropriate number $N$ of the coefficients to be computed is performed using \eqref{Eqsn}. Indeed, the sufficient smallness of the expressions
    \begin{align}\label{IndDirect}
                    \varepsilon_{1,N}&=\max_{\zeta} \left| \sum_{n=0}^{N-1} g_n(\zeta) - \sum_{n=0}^{N-1} s_n(\zeta) \right|, 
\end{align}
\begin{equation*} \varepsilon_{2,N}=\max_{\zeta} \left| \sum_{n=0}^{N-1} \frac{g_n(\zeta)}{\zeta}- \omega(\zeta) \right|\quad \textrm{or}\quad
        \varepsilon_{3,N}=\max_{\zeta} \left| \sum_{n=0}^{N-1} \frac{s_n(\zeta)}{\zeta}- \omega(\zeta)\right|, 
\end{equation*}
when $\zeta=\delta$, indicates a sufficiently good approximation of the characteristic function. 
\end{rem}

\begin{rem} \label{RemarkPrincipleAlg}
   We locate the zeros of the approximate characteristic function \eqref{eq:NSBFDirectEqs-1} with the aid of the argument principle theorem \cite{conway1978}. In particular, we compute the change of the argument along rectangular contours $\gamma$. If the change of the argument along $\gamma$ is zero, then consider another contour. Otherwise, subdivide the region within the contour until getting the desired accuracy.  See \cite{VKSTUV} for a detailed application of this methodology.
\end{rem}

The accurate enough computation of the zeros of $D_{0}(k)$ in
some strip $\left\vert\operatorname{Im} k\right\vert <C$ with the aid of the
argument principle (Remark \ref{RemarkPrincipleAlg}) is guaranteed by the accurate enough uniform
approximation of $D_{0}(k)$ by $D_{0,N}(k)$ in the same strip (see Theorem \ref{Th NSBF old})
and by the following version of Rouch\'{e}'s theorem.

\begin{thm}(\cite[p. 213]{Dettman})
   \label{Rouche} 
   Let $f(k)$ and
$g(k)$ be analytic functions within and on a simple closed contour $\gamma$
which satisfy the inequality $\left\vert g(k)\right\vert <\left\vert
f(k)\right\vert $ on $\gamma$, where $f(k)$ does not vanish. Then $f(k)$ and
$f(k)+g(k)$ have the same number of zeros inside $\gamma$. 
\end{thm}

Hence, we obtain:

\begin{prop}\label{approxProp}
   Assume that the truncation parameter $N$ is chosen so that
   \[
      \lvert D_{0}(k) - D_{0,N}(k)\rvert < \varepsilon
      \quad \text{whenever } \left\vert\operatorname{Im} k\right\vert <C.
   \]
   Let $\gamma$ be any simple closed contour belonging to the strip $\left\vert\operatorname{Im} k\right\vert <C$, and suppose
   \[
      \varepsilon < \min_{\,k\in\gamma} \lvert D_{0,N}(k)\rvert.
   \]
   Then $D_{0}(k)$ and $D_{0,N}(k)$ have the same number of zeros inside $\gamma$.
\end{prop}

\begin{proof}
   For $k$ in the strip $\left\vert\operatorname{Im} k\right\vert <C$, we have
   \[
      \lvert D_{0}(k) - D_{0,N}(k)\rvert < \varepsilon.
   \]
   In particular, on any simple closed contour $\gamma$ in that strip, if 
   \[
      \varepsilon < \min_{\,k\in\gamma} \lvert D_{0,N}(k)\rvert,
   \]
   then
   \[
      \lvert D_{0}(k) - D_{0,N}(k)\rvert 
      < \min_{\,k\in\gamma} \lvert D_{0,N}(k)\rvert 
      \quad \text{for all } k \in \gamma.
   \]
   Denote
   \[
      f(k) := D_{0,N}(k),
      \qquad
      g(k) := D_{0}(k) - D_{0,N}(k).
   \]
   By construction, both $f(k)$ and $g(k)$ are analytic on and inside $\gamma$.  Moreover, on $\gamma$,
   \[
      \lvert g(k)\rvert = \lvert D_{0}(k) - D_{0,N}(k)\rvert 
      < \min_{\,k\in\gamma} \lvert D_{0,N}(k)\rvert 
      = \min_{\,k\in\gamma} \lvert f(k)\rvert,
   \]
   so that
   \[
      \lvert g(k)\rvert < \lvert f(k)\rvert 
      \quad \text{for all } k \in \gamma,
   \]
   and $f(k)$ does not vanish on $\gamma$.  By Theorem \ref{Rouche}, $f(k)$ and $f(k)+g(k) = D_{0}(k)$ have the same number of zeros inside $\gamma$, and this completes the proof.
\end{proof}

\begin{rem} \label{RemarkSplines}
Locating only the real zeros of the approximate characteristic function \eqref{eq:NSBFDirectEqs-1} can be performed by interpolating this function in the real axis with a spline, and finding its roots. 
\end{rem}

\subsection{Inverse transmission eigenvalue problem} \label{sec_comp_nsbf_inv}

Next, we consider the application of the NSBF representations to solve the inverse problem.  This discussion is divided into three parts: the recovery of $\delta$ from transmission eigenvalues, the recovery of $n(r)$ from transmission eigenvalues, and the application of spectrum completion in addressing the inverse problem.

\subsubsection{Recovery of $\delta$ from transmission eigenvalues} \label{sec_comp_nsbf_inv_delta}

In inverse Sturm-Liouville eigenvalue problems for (\ref{eq:SL}), $\delta$ is usually known, as it corresponds to the right endpoint of the domain of definition. Nevertheless, in the more general category of coefficient inverse problems involving string-type Sturm-Liouville equations, the application of the Liouville transformation introduces an unknown transformed parameter. Developing accurate computational methods to estimate this parameter is therefore of significant interest and has broad applications in various inverse problems. We refer, e.g., to \cite{Mirzaei}, where this problem is discussed in application to the inverse two-spectrum problem for the string equation. Here, we treat the recovery of $\delta$ as a necessary first step of our approach to the reconstruction of the refractive index from transmission eigenvalues. 

Let us discuss different methodologies for recovering $\delta$ from transmission eigenvalues. For the inverse TEP (\ref{eq:MainSL})-(\ref{eq:CharEq}), $\delta$ can be estimated from the real sub-spectrum using the asymptotic formula \cite[Lemma 2]{McP}:
\begin{equation}
    k_j^2=\frac{j^2\pi ^2}{(\delta -1)^2}+O(1),\quad j\rightarrow +\infty, \label{aformula}
\end{equation}
provided that $\delta\neq 1$ and the refractive index satisfies $n\in C^1(\mathbb{R}),\ n^{\prime\prime}\in \mathcal{L}^2[0,1]$ and $n(1)=1,\ n^{\prime}(1)=0$.
We refer to \cite{MPS} for an example of using the above formula to recover $\delta$ from the knowledge of the lowest real eigenvalues.

Furthermore, if both $n(1)\neq 1$ and $\delta \neq 1$, it is shown in \cite{CLM} that the density of all (real and complex) zeros of $D_0$ in the right half-plane is equal to $(\delta +1)/\pi$. Consequently, $\delta$ can be uniquely determined by the knowledge of all transmission eigenvalues. This result is further explored in \cite[Section 7.6.3]{KBook}. We can use this density argument, to approximate $\delta\neq 1$ when $n(1)\neq 1$. Specifically, given the number of eigenvalues $\mathcal{N}$ within the strip $0<\operatorname{Re} k<\mathcal{R}$, for $\mathcal{R}$ large enough, the average density of the roots is equal to $\mathcal{N}/\mathcal{R}$. Therefore, $\delta$ can be approximated by 
\begin{equation}
      \tilde{\delta}=\frac{\mathcal{N}\pi}{\mathcal{R}}-1,\quad \mathcal{R}>>0. \label{aformula2}  
\end{equation}

The above approaches have some drawbacks due to the specific restrictions on the values of $\delta, \ n(1)$ and $n^\prime(1)$. To overcome these limitations, we develop a novel method based on the NSBF formulas. This is achieved by combining the expressions for the coefficients $s_n(\delta)$ and $g_n(\delta)$ with the indicator function in \eqref{IndDirect}, and is presented in Algorithm \ref{algo_delta} that follows. More specifically, we assume knowledge of a set of transmission eigenvalues $\{ k_{j} \}_{j=1}^{J}$, as well as the values of $n(1)$ and $n'(1)$. Then, $\delta$ and $N$ are estimated by solving a system of algebraic equations for the coefficients $s_n(\delta)$ and $g_n(\delta)$. The corresponding  \textquotedblleft optimal\textquotedblright\ values are determined by $\underset{\delta, N}{\textrm{argmin}}\ \varepsilon$, see below Remark~\ref{RemarkEpsilonAtDelta}. 

\begin{rem} \label{RemarkEpsilonAtDelta}
Based on \eqref{IndDirect}, the choice of an appropriate number of the coefficients to be computed and approximate $\delta^* \approx \delta$ is performed by choosing the values minimizing 
\begin{align*}
                     \varepsilon_{1,N}(\delta)&= \left| \sum_{n=0}^{N-1} g_n(\delta) - \sum_{n=0}^{N-1} s_n(\delta) \right|,
\end{align*}
thus   \begin{align*}
                    \delta^* = \underset{\delta, N}{\text{argmin}}\ \varepsilon_{1,N}(\delta),
\end{align*}
see Algorithm \ref{algo_delta} for more details.
\end{rem}

\begin{algorithm}
\caption{Computation of $\delta$.}
\label{algo_delta}
\begin{minipage}[b]{1\textwidth}
Assume that the set of transmission eigenvalues $\left\{ k_{j}\right\} _{j=1}^{J},$ $n(1)$ and $n'(1)$ are given. 

\begin{enumerate}[itemsep=0.1ex, topsep=0ex]
\item For a grid of points  $\{\delta^m\}_{m=1}^M$ and an array of values of $N$, consider the equations $D_{0,N}(k_j)=0$ (see \eqref{eq:NewCharEqTrunca})
\begin{align}
 & a(k_j)\sum_{n=0}^{N-1}g_{n}(\delta^m)j_{2n}(k_{j}\delta^m)+
\frac{b(k_j)}{k_j} \sum_{n=0}^{N-1}s_{n}(\delta^m)j_{2n+1}(k_{j}\delta^m)\nonumber \\
& =-a(k_j)\text{cos}(k_{j}\delta^m)-b(k_j)\frac{\text{sin}(k_{j}\delta^m)}{k_{j}}.\label{eq:MainApprox3}
\end{align}
\item For each $\delta^m$, solve the finite system (\ref{eq:MainApprox3}) of $J$ linear
algebraic equations for the coefficients $\left\{ g_{n}(\delta^m)\right\} _{n=0}^{N-1}$
and $\left\{ s_{n}(\delta^m)\right\} _{n=0}^{N-1}$ , where $2N\leq J$ (so we have square or overdetermined systems).

\item Choose the number of the NSBF coefficients $N^*$ which delivers  $\underset{\delta^m, N}{\text{min }} \varepsilon_{1,N}(\delta^m)$, see Remark~\ref{RemarkEpsilonAtDelta}.

\item Using $N^*$ find $\delta^*=\underset{\delta, N^*}{argmin} \  \varepsilon_{1,N^*}(\delta)$, by refining the initial mesh of points $\delta$. 

\end{enumerate}
\end{minipage}
\end{algorithm}

\subsubsection{Recovery of $n(r)$ from transmission eigenvalues} \label{sec_comp_nsbf_inv_index}

We now turn our attention to the inverse problem, that is recovering the refractive index from the knowledge of transmission eigenvalues. The question of the uniqueness of this inverse eigenvalue problem has been the subject of investigation for many years up to the present. For further details, see Chapter 6 of \cite{CCH}, Chapter 10.4 of \cite{CKbook}, Section 4 of \cite{Pallreview}, and the associated references.

It can be easily verified that the refractive index and $r(\zeta)$ satisfy the following initial value problems
\begin{align}
-\ddot{(n^{1/4})}+p(\zeta)n^{1/4} & =0,\,0<\zeta<\delta, \label{ivp1}
\end{align}
\begin{align}
n^{1/4}(r(0))=n^{1/4}(1):=n_{0},\,\left.\frac{d}{d\zeta}n^{1/4}(r(\zeta))\right|_{\zeta=0}:=n_{1}, \label{ivp2}
\end{align}
and \begin{equation}\frac{dr}{d\zeta}=-\frac{1}{\sqrt{n(r(\zeta))}},\quad r(0)=1.\label{inverseR}\end{equation} 

When using the transformation (\ref{Ltrans}), analogous initial value problems are defined \cite{CCqual, KBook}. Given the potential $p(\zeta)$, the values $n(1),\ n'(1)$ and $\delta$, the refractive index $n(r)$ can be  uniquely determined. For further details on the unique determination of $n(r)$ from an appropriate definition of a Goursat problem for $p(\zeta)$, see, for example, \cite{CL}.

Now, from (\ref{ivp1})-(\ref{ivp2}), we obtain that 
\begin{align}\label{MainEq_n}
n^{1/4}(r(\zeta))=n_{0}\phi(0,\zeta)+n_{1}S(0,\zeta). 
\end{align}
Therefore, using the relations in \eqref{eq:g0s0}, the function $n(r)$ can be
written in terms of the first NSBF coefficients $g_{0}(\zeta)$ and $s_{0}(\zeta)$, 
\begin{equation}
n^{1/4}(r(\zeta))=n_{0}(g_{0}(\zeta)+1)+n_{1}\left(\frac{s_{0}(\zeta)}{3}+1\right)\zeta.\label{eq:RecoverR}
\end{equation}
Moreover, to solve the Cauchy problem \eqref{inverseR} that allows to recover the refractive index in the original interval, equation \eqref{eq:RecoverR} is used to express $r(\zeta)$ in terms of the first coefficients $g_{0}(\zeta)$ and $s_{0}(\zeta)$ as follows
\begin{equation}
r(\zeta)=1-\int_{0}^{\zeta}\frac{1}{\sqrt{n(r(t))}}dt=1-\int_{0}^{\zeta}\frac{1}{\left(n_{0}(g_{0}(t)+1)+n_{1}\left(\frac{s_{0}(t)}{3}+1\right)t\right)^{2}}dt,\label{eq:ChangeVariable}
\end{equation}
 which can be simplified by using the following observations.
 
\begin{rem}
Note that the solution $S(0,\zeta)$ of \eqref{ivp1} can be obtained by applying the Abel formula to the solution $n^{1/4}(r(\zeta))$:
\begin{align*}  
S(0,\zeta)=n_0n^{1/4}(r(\zeta))\int_{0}^{\zeta}\frac{1}{\sqrt{n(r(t))}}dt,
\end{align*}
which can be rewritten as
\begin{align}\label{S2}  
S(0,\zeta)=n_0n^{1/4}(r(\zeta))(1-r(\zeta))),
\end{align}
by using the first equality in \eqref{eq:ChangeVariable}. 
\end{rem}
Substitution of formula \eqref{eq:RecoverR} into \eqref{S2} gives another expression for $r(\zeta)$ in terms of the coefficients $g_0(\zeta)$ and $s_0(\zeta)$,
\begin{equation}\label{NewChangeVariable}
    r(\zeta)=1-\frac{\left(s_0(\zeta)+3\right)\zeta}{3n_0^2(g_0(\zeta)+1)+n_0n_1(s_0(\zeta)+3)\zeta}.
\end{equation}
In the special case $n_{1}=0$ (i.e. $n^\prime(1)=0$), the above expression for $r(\zeta)$ simplifies to
\begin{equation}
   r(\zeta)=1-\frac{\left(s_0(\zeta)+3\right)\zeta}{3n_0^2(g_0(\zeta)+1)}.\label{eq:-2}
\end{equation}
Indeed, \eqref{eq:-2} can also be obtained by applying the result from Remark \ref{RemarkAbelForS} to equation \eqref{eq:ChangeVariable}.

 \begin{prop}
 The solutions $S(0,x)$ and $\phi(0,x)$ satisfy the relations

\begin{align}\label{S3}
S(0,\delta)=n_0n(0)^{1/4},
\end{align} 
\begin{align}\label{AuxPhi}
\phi(0,\delta)=\frac{n^{1/4}(0)(1-n_0n_1)}{n_0}
\end{align}
and
\begin{align}\label{AuxS_Phi}
S(0,\delta)(1-n_0n_1)=n^2_{0}\phi(0,\delta). 
\end{align}
\end{prop}
\begin{proof}
Considering $\zeta=\delta$ in the equation \eqref{S2} leads to \eqref{S3}.
Additionally, from \eqref{MainEq_n} we have
\begin{align}\label{MainEq_n2}
n^{1/4}(0)=n_{0}\phi(0,\delta)+n_{1}S(0,\delta). 
\end{align}
Substitution of \eqref{S3} into \eqref{MainEq_n2} leads to \eqref{AuxPhi} and \eqref{AuxS_Phi}.
\end{proof}
The relation \eqref{AuxS_Phi} yields to a useful formula between the first coefficients $g_0(\delta)$ and $s_0(\delta)$
\begin{equation}\label{relCoefs}
    g_0(\delta)=s_0(\delta) \frac{\delta(1-n_0n_1)}{3n_0^2}+\frac{\delta(1-n_0n_1)}{n_0^2}-1.
\end{equation}
With the aid of \eqref{relCoefs} we rewrite system \eqref{eq:MainApprox3} in order to eliminate $g_0(\delta)$ and reduce the number of unknowns

\begin{align}
&s_{0}(\delta)\left( \frac{b(k_j)}{k_j}j_{1}(k_{j}\delta)+a(k_j)\frac{\delta (1-n_0n_1)}{3n_0^2}j_{0}(k_{j}\delta)\right)+a(k_j)\sum_{n=1}^{N-1}g_{n}(\delta)j_{2n}(k_{j}\delta) \nonumber\\
&+\frac{b(k_j)}{k_j}\sum_{n=1}^{N-1}s_{n}(\delta)j_{2n+1}(k_{j}\delta) =-a(k_j)\left(\text{cos}(k_{j}\delta)+\left(\frac{\delta(1-n_0n_1)}{n_0^2}-1\right)j_0(k_j\delta)\right) \nonumber\\
&-b(k_j)\frac{\text{sin}(k_j\delta)}{k_j}. 
\label{NewSystemReduced}
\end{align}

From the discussion above, we can devise an algorithm to solve the inverse problem. The detailed steps are presented in Algorithm \ref{algo_inv}. Note that the fact that zero is a transmission eigenvalue is equivalent to \eqref{AuxS_Phi}, which is considered when solving the inverse problem by using system \eqref{NewSystemReduced}. 

\begin{algorithm}
\caption{The inverse transmission eigenvalue problem.}

\label{algo_inv}
\begin{minipage}[b]{1\textwidth}
Assume that the set of nonzero eigenvalues  $\left\{ k_{j}\right\} _{j=1}^{J},$ $n(1)$ and $n'(1)$ are given. 

\begin{enumerate}[itemsep=0.1ex, topsep=0ex]
 \item If $\delta$ is known, use Remark \ref{RemarkEpsilon} to find an optimal $N$.
 \item If $\delta$ is unknown, use Algorithm \ref{algo_delta}.
\item  Consider the equations $D_{0,N}(k_j)=0$ 
 by using  (\ref{NewSystemReduced}).

\item Solve the finite system (\ref{NewSystemReduced}) of $J$ linear
algebraic equations for the coefficients $\left\{ g_{n}(\delta)\right\} _{n=1}^{N-1}$
and $\left\{ s_{n}(\delta)\right\} _{n=0}^{N-1}$ , where $2N\leq J+1$. Find $g_0(\delta)$ from equation \eqref{relCoefs}, and $n(0)$ from \eqref{S3}. 

\item Construct the approximate solutions $\phi_{N}(k,\delta)$ and $S_{N}(k,\delta)$,
for $k$ in a strip of the complex plane, from the sets of coefficients found in the previous
step and (\ref{eq:STrn})-(\ref{eq:PhiTrun}). 

\item Approximate the identity (\ref{eq:Identity}) by using (\ref{eq:STrn}), (\ref{eq:PhiTrun}) and (\ref{eq:TTrun}) and
the expressions for $\phi_{N}(k,\delta)$ and $S_{N}(k,\delta)$ of step 5, 
i.e.,
\[
T_{N}(k,\zeta)=\phi_{N}(k,\delta)S_{N}(k,\zeta)-\phi_{N}(k,\zeta)S_{N}(k,\delta),\quad 0<\zeta<\delta
\]
which is equal to 
\begin{align}
 & \frac{\text{sin}(k(\zeta-\delta))}{k}-\phi_{N}(k,\delta)\frac{\text{sin}(k\zeta)}{k}+S_{N}(k,\delta)\cos(k\zeta)=-\frac{1}{k}\sum_{n=0}^{N-1}t_{n}(\zeta)j_{2n+1}(k(\delta-\zeta))\nonumber \\
 & -S_{N}(k,\delta)\sum_{n=0}^{N-1}g_{n}(\zeta)j_{2n}(k\zeta)+\frac{\phi_{N}(k,\delta)}{k}\sum_{n=0}^{N-1}s_{n}(\zeta)j_{2n+1}(k\zeta).\label{eq:LAstSystem}
\end{align}

\item Solve the finite system of $M$ linear algebraic equations for
the coefficients $\left\{ t_{n}(\zeta)\right\} _{n=0}^{N-1}$, $\left\{ g_{n}(\zeta)\right\} _{n=0}^{N-1}$
and $\left\{ s_{n}(\zeta)\right\} _{n=0}^{N-1}$ constructed from equation
(\ref{eq:LAstSystem}) evaluated at a set of distinct points $k=\left\{ k_{n}\right\} _{n=1}^{M}$. 

\item Recover $n(r)$ for $r\in(0,1)$ by substituting the coefficients $g_0(\zeta)$ and $s_0(\zeta)$ found in the previous step in equations (\ref{eq:RecoverR}) and (\ref{NewChangeVariable}), or (\ref{eq:-2}) if $n'(1)=0$.

\end{enumerate}
\end{minipage}
\end{algorithm}

\subsubsection{Spectrum completion and the inverse problem} \label{sec_comp_nsbf_inv_scompl}

In this section, we study the spectrum completion for the transmission eigenvalue problem. That is, given a small subset of real and/or complex eigenvalues, we aim to compute subsequent eigenvalues with high accuracy. Using this method, we proceed with solving the inverse problem, having only a few eigenvalues as input for the algorithm. This idea was originally introduced in \cite{KravSC} for the direct and inverse Sturm-Liouville eigenvalue problem, to which we refer for further details; see also \cite{kravchenko2025sinc}. 

The possibility of the spectrum completion is based on Proposition \ref{approxProp}, which
essentially establishes that if a sufficient number of the NSBF coefficients
is recovered accurately enough, then zeros of the approximate characteristic
function are close to those of the exact one. Moreover, as we show in
Subsection \ref{sec_exampl_inv_scompl}, application of this technique previous to solving the
inverse problem may contribute in stabilizing the result of the
reconstruction, especially when the number of the originally given eigenvalues
was small. This idea of using the \textquotedblleft
completed\textquotedblright \ eigenvalues is akin to completing the set of
given eigenvalues by the asymptotic ones, which is frequently used in inverse
spectral problems, see, e.g., the discussion in \cite[Section 13.3]{krBook2020}.
However, usually the use of the asymptotic eigenvalues  requires some
additional information on the unknown coefficient of the equation. Spectrum
completion is free of this drawback.

Thus, the solution of the inverse problem combined with the spectrum completion involves three main steps: first, completing the spectra; second, reducing the problem to a system of linear algebraic equations; and finally, reconstructing the refractive index and solving the inverse problem. To complete the spectrum, we follow the steps summarized in Algorithm \ref{algo_sc} and then solve the inverse problem using Algorithm \ref{algo_inv}.

\begin{algorithm}
\caption{Spectrum completion.}
\label{algo_sc}
\begin{minipage}[b]{1\textwidth}
Assume that a set of eigenvalues $\left\{ k_{j}\right\} _{j=1}^{J}$
, $n(1),$ $n'(1)$ and possibly $\delta$ are given.

\begin{enumerate}[itemsep=0.1ex, topsep=0ex]
\item Perform steps  1-5 from Algorithm \ref{algo_inv}.
\item Construct $D_{0,N}(k)$ for $k$ in a strip of the complex plane.
\item Locate more zeros of $D_{0,N}(k)$ by using Remarks \ref{RemarkPrincipleAlg} and \ref{RemarkSplines}. 

\end{enumerate}
\end{minipage}
\end{algorithm}

 \section{Numerical Examples} \label{sec_exampl}

In the following, we present various examples to verify the validity of our NSBF approximation methodology, for solving both the direct and the inverse problems. 

The computations were performed on a standard desktop computer with Intel Core i3 (2.00GHz) computer, with 12GB RAM, using MATLAB version 2024b. The average computational time was a few minutes for solving the direct problems, a few seconds for computing  $\delta$, and approximately a couple of seconds for the inverse problems. These times varied depending on the complexity of each problem considered.

 \subsection{Direct transmission eigenvalue problem} \label{sec_exampl_dir}

We begin by presenting numerical examples for the direct problem, aimed at validating the NSBF approximation methodology. The Algorithm \ref{algo_dir} described in Section \ref{sec_comp_nsbf_dir} is applied in the examples that follow. 

\begin{example} \label{exam1}
Let $n(r)=16/\left((r+1)(3-r)\right)^{2}$. The corresponding potential under the Liouville transformation is
$\ p(\zeta(r))=1/4$, while the new variable $\zeta$ lies in the interval $[0,\log(3)]$. By minimizing the indicator $\varepsilon_{1,N}$ in \eqref{IndDirect}, we obtain $N=6$, corresponding to $\displaystyle \min_{N\in[1,50]} \varepsilon_{1,N}=4.22\times 10^{-15}$, as presented in Figure \ref{ex1} (right). 
The  respective calculated eigenvalues are shown in Figure
\ref{ex1} (left). Our results are in agreement with \cite[Example 2]{CLM}. 

\begin{figure}[H]
\begin{centering}
\begin{minipage}[c][1\totalheight][t]{0.48\textwidth}%
\centering
\includegraphics[scale=0.58]{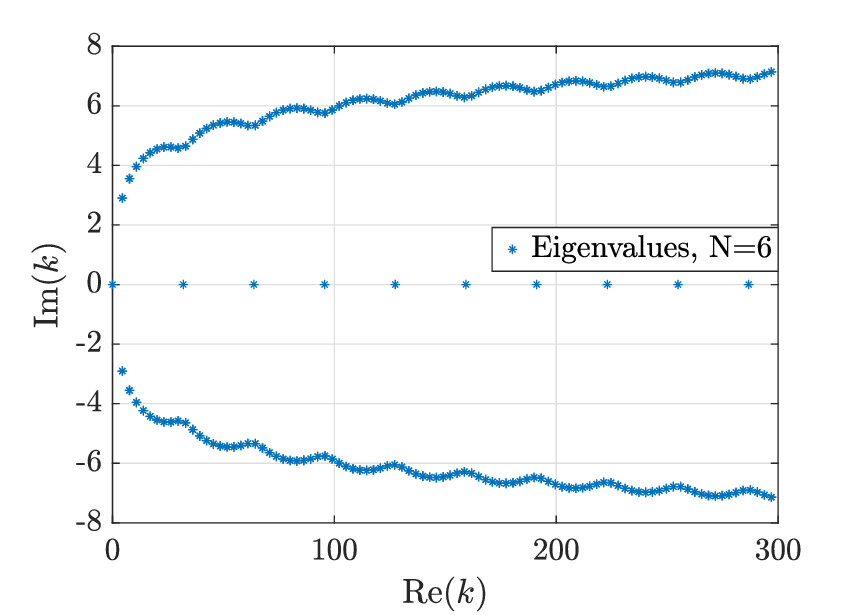}
\par
\end{minipage}
\begin{minipage}[c][1\totalheight][t]{0.48\textwidth}%
\centering
\includegraphics[scale=0.58]{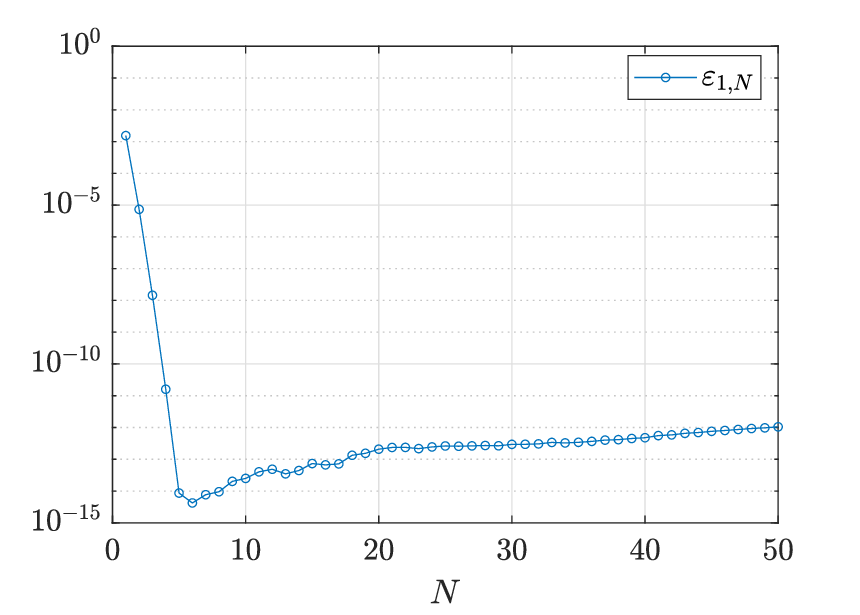}
\par
\end{minipage}
\caption{Real and complex transmission eigenvalues (left) and indicator $\varepsilon_{1,N}$ (right) for $n(r)=16/\left((r+1)(3-r)\right)^{2}$ of Example \ref{exam1}.}
\label{ex1}
\par\end{centering}
\end{figure}

For this example, it is possible to obtain the characteristic equation in closed form. We solved in high precision the direct problem using the root-finding function in Mathematica 11 applied to the closed-form characteristic equation and compared these results to those obtained using our approximation method. This gave us the maximum absolute error of $2.31\times 10^{-7}$ of the transmission eigenvalues presented in Figure \ref{ex1}, demonstrating the accuracy and validity of our method. 
\end{example}

\begin{example}\label{exam2} We consider the three refractive
indices $n_{1}(r)$, $n_{2}(r)$ and $n_{3}(r)$ presented in \cite[Section 5]{CL2}
with the characteristic that they have equal $\delta=\pi/4$ and their real eigenvalues are very close:
\[
n_{1}(r)=\frac{1}{(1+(1-r)^{2})^{2}},\,\,n_{2}(r)=\left(\frac{\pi}{4}\right)^{2},\,\,n_{3}(r)=\left(1+0.4292\left(r-1\right)\right)^{2}.
\]
The potentials under the Liouville transform are $p_{1}(\zeta(r))=-1$, $p_{2}(\zeta(r))=0$
and $p_{3}(\zeta(r))=-4.0714/(1.3299+r)^{4}$. Following a similar approach to the previous example, we minimize (\ref{IndDirect}) to derive $N=7$, $N=3$  and $N=23$ for $p_1, p_2$ and $p_3$ respectively. These in turn correspond to the errors $\displaystyle \min_{N\in[1,50]} \varepsilon_{1,N}=5.22\times 10^{-15}$, $\displaystyle \min_{N\in[1,50]} \varepsilon_{1,N}=5.1\times 10^{-15}$ and $\displaystyle \min_{N\in[1,50]} \varepsilon_{1,N}=4.16\times 10^{-16}$. Figure \ref{fig:EigenvalsColtonLeung2A2017} presents the eigenvalues found in each case. We observe that the real eigenvalues of all refractive indices are close, while the complex eigenvalues exhibit different distributions for each case.

\begin{figure}[H]
\begin{centering}
\begin{minipage}[c][1\totalheight][t]{0.48\textwidth}%
\centering
\includegraphics[scale=0.55]{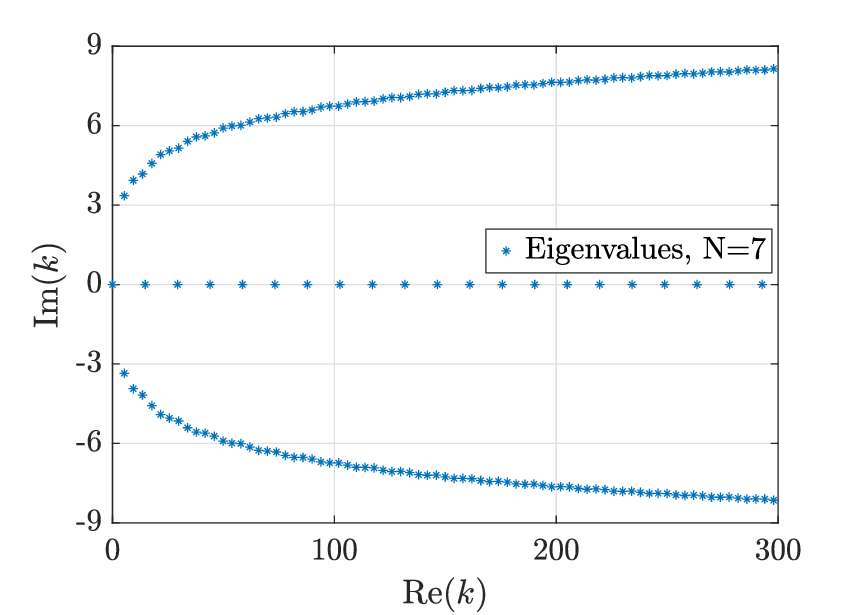}
\end{minipage}
\begin{minipage}[c][1\totalheight][t]{0.48\textwidth}%
\centering
\includegraphics[scale=0.55]{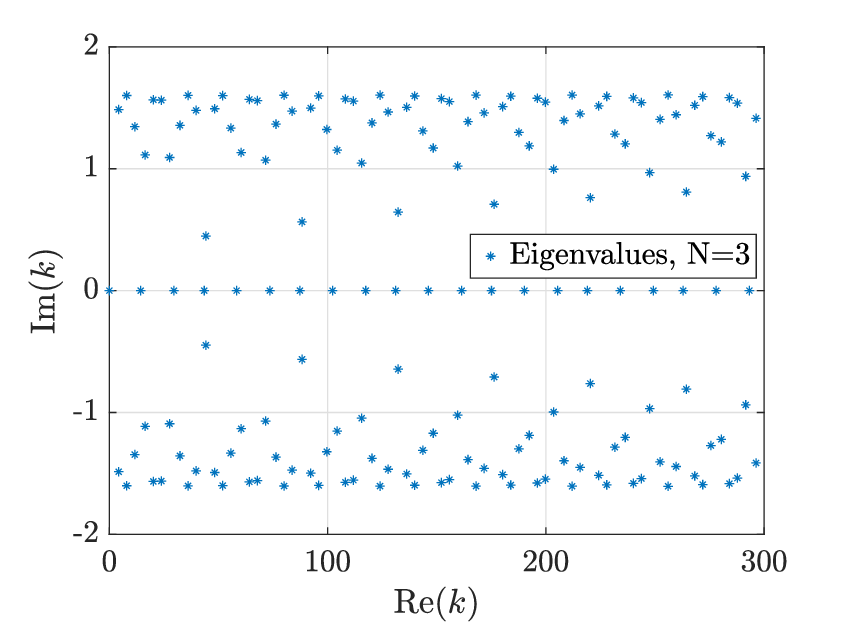}
\end{minipage}\\[2mm] 
\begin{minipage}[c][1\totalheight][t]{0.48\textwidth} 
\centering
\includegraphics[scale=0.55]{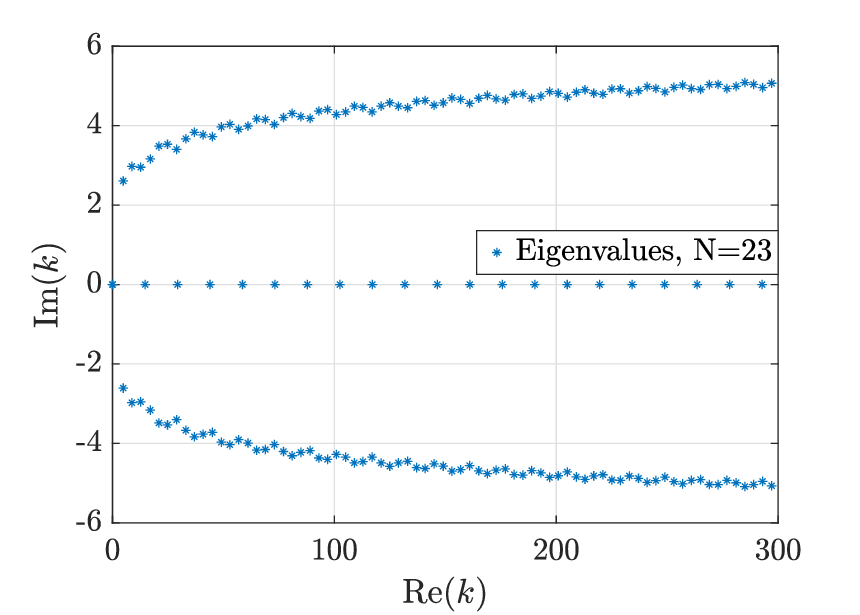}
\end{minipage}
\par\end{centering}
\caption{\label{fig:EigenvalsColtonLeung2A2017}Real and complex transmission eigenvalues corresponding to the refractive indices 
$n_{1}(r)=1/(1+(1-r)^{2})^{2}$ (top left), $n_{2}(r)=(\pi/4)^{2}$ (top right), and
$n_{3}(r)=(1+0.4292(r-1))^{2}$ (bottom) of Example~\ref{exam2}.}
\end{figure}
\end{example}

\begin{example} \label{exam3}
Consider the case of the refractive index \textbf{ $n(r)=1.2 + (1 - r)\sin(2\pi r)$}. The corresponding potential under the Liouville transformation is
\begin{align*}
    p(\zeta(r))&=-\frac{5 (\sin (2 \pi  r)+2 \pi  (r-1) \cos (2 \pi  r))^2}{16 \left(\frac{6}{5}-(r-1) \sin (2 \pi  r)\right)^3}\\
    &-\frac{ \pi  (5 (r-1) \sin (2 \pi  r)-6) (\pi  (r-1) \sin (2 \pi  r)-\cos (2 \pi  r))}{5 \left(\frac{6}{5}-(r-1) \sin (2 \pi  r)\right)^3}
\end{align*}
where variable $\zeta\in[0,\delta]$ with $\delta \approx 1.155384328946918$. We obtain $N=41$, according to $\displaystyle \min_{N\in[1,50]} \varepsilon_{1,N}=1.03\times 10^{-13}$. The estimated eigenvalues are shown in Figure
\ref{ex3}.

\begin{figure}[H]
\begin{centering}
\includegraphics[scale=0.6]{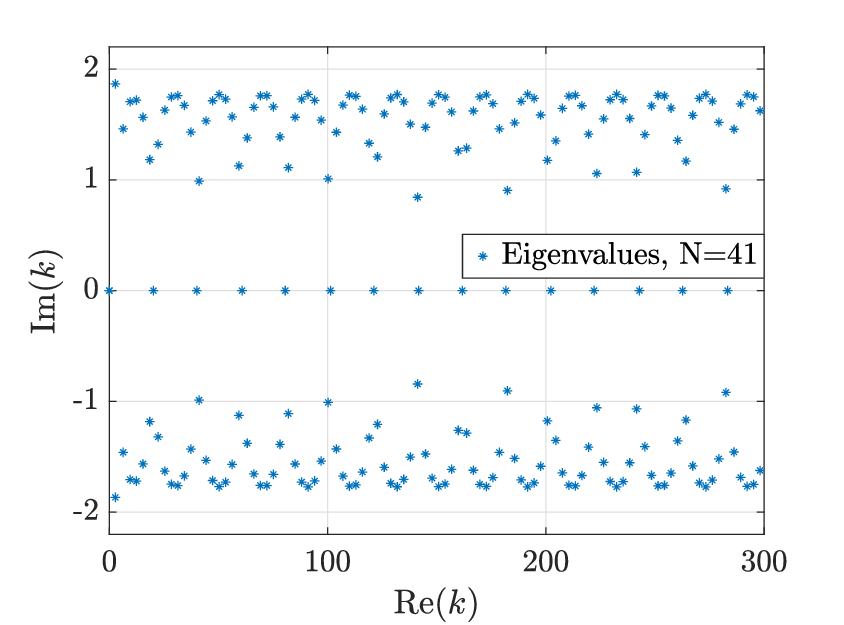}
\par\end{centering}
\caption{Real and complex transmission eigenvalues for $n(r)=1.2 + (1 - r)\sin(2\pi r)$ of Example \ref{exam3}.}
\label{ex3}
\end{figure}

\end{example}

\begin{example}\label{ExampleDelta1}
Assume that \textbf{$n(r)=(r+0.5)^2$}. Then, the corresponding potential
$p(\zeta(r))=-12/(1+2r)^4$ is defined for $\zeta\in[0,\delta]$ with $\delta =1$. The value $N=25$, is derived according to $\displaystyle \min_{N\in[1,50]} \varepsilon_{1,N}=4.44\times 10^{-16}$. Figure \ref{EigensExampleDelta1} presents the calculated transmission eigenvalues. We notice that no complex eigenvalues appear, consistent with the observations made in \cite{CL2} for similar types of refractive indices. 

\begin{figure}[H]
\begin{centering}
\includegraphics[scale=0.6]{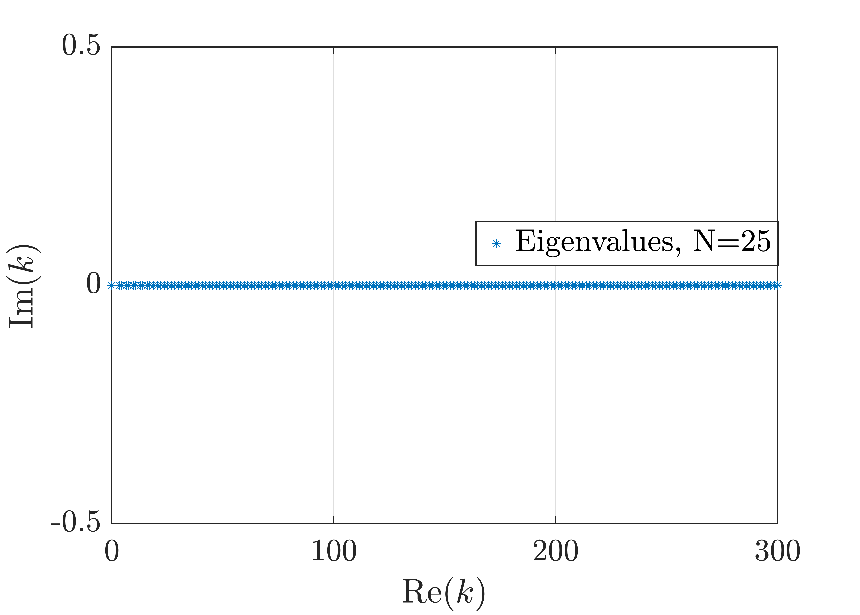}
\par\end{centering}
\caption{Real and complex transmission eigenvalues for $n(r)=(r+0.5)^2$ of Example \ref{ExampleDelta1}.}
\label{EigensExampleDelta1}
\end{figure}
\end{example} 

In all the examples studied above, we observe that real and complex eigenvalues can be computed. This includes as many eigenvalues as needed, even those with high magnitudes. Additionally, we notice that the complex eigenvalues may lie within a fixed strip parallel to the real axis, depending on the values of $\delta,\ n(1)$ and $n^{\prime}(1)$. Such behavior aligns with the discussion in \cite{CL,CLM}.

\subsection{Inverse transmission eigenvalue problem} \label{sec_exampl_inv}

Next, we present numerical examples for the inverse problem, solved using the approach described in Section \ref{sec_comp_nsbf_inv} and the corresponding algorithms. Specifically, we first recover $\delta$, and then reconstruct the unknown refractive index using a few eigenvalues of smallest magnitude. Finally, we also examine the use of spectrum completion and its potential application to the inverse problem. 

In all examples studied, we use a number of input eigenvalues, which are ordered by ascending real part. For the non-real eigenvalues, since they occur in complex conjugate pairs, we require the knowledge of only one member of each pair; its conjugate counterpart is then included in the input by conjugation. Thus, when we say that $J$ eigenvalues are given, this means that this set of $J$ eigenvalues also includes conjugate ones, if any.

\subsubsection{Computation of $\delta$} \label{sec_exampl_inv_delta}

In the following examples, we demonstrate the application of Algorithm \ref{algo_delta} for recovering $\delta$ from transmission eigenvalues. Note that any of the systems \eqref{eq:MainApprox3} or \eqref{relCoefs}-\eqref{NewSystemReduced} can be used in this procedure. We also compare the results with the asymptotic formulas given in (\ref{aformula}) or (\ref{aformula2}), where applicable. 

To perform Algorithm \ref{algo_delta} we consider an initial coarse mesh of $\delta$'s and values of $N$ for which by solving the linear system \eqref{eq:MainApprox3} and finding the minimum as in Remark \ref{RemarkEpsilonAtDelta} we choose a suitable value of $N$. For simplicity and readability, we shall henceforth use the notation $N$ and $\delta$ instead of $N^*$ and $\delta^*$, respectively. Next, once the value of $N$ is fixed, a number of refinements in the grid for $\delta$ are made according to the criterion presented in Remark \ref{RemarkEpsilonAtDelta}. Thus the output of the algorithm is the value of $N$ computed in the first iteration and the value of $\delta$ from the last iteration. 

\begin{example} \label{exam5}
We recover $\delta$ for the refractive index considered in Example \ref{exam1}, using eigenvalues computed from the NSBF representations. The unknown $\delta$ can be estimated from the lowest 10 real eigenvalues using formula (\ref{aformula}), as demonstrated in Figure~\ref{delta_ex5}. The resulting absolute error is $1.35\times 10^{-6}$.  

\begin{figure}[H]
        \centering
        \includegraphics[scale=0.55]{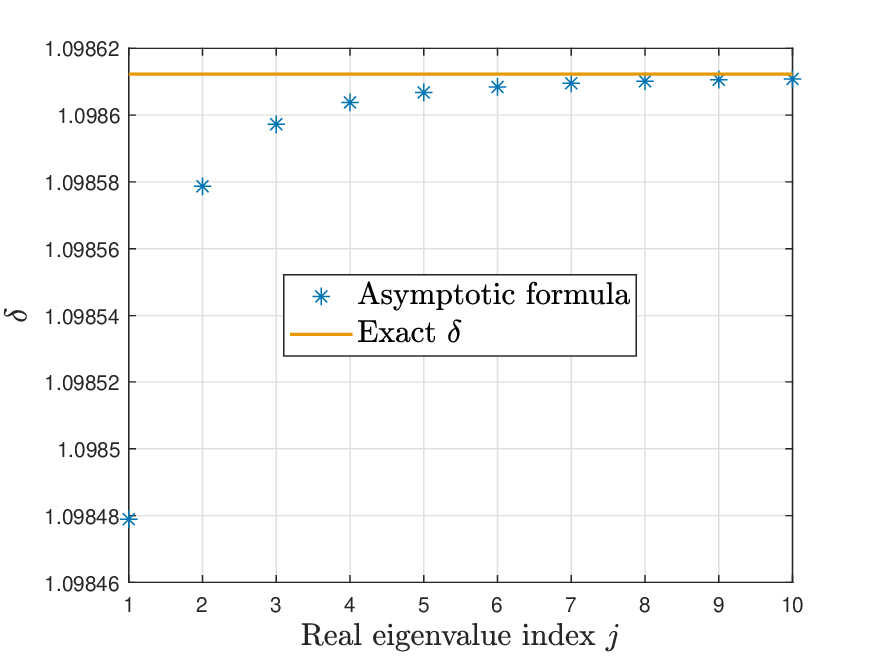}
    \caption{ Asymptotic formula for $\delta$ of Example \ref{exam5}.}
    \label{delta_ex5}
\end{figure}

Besides this approach, we approximate $\delta$ applying Algorithm \ref{algo_delta}. The input data are $10$ lowest magnitude (complex) transmission eigenvalues and the array $N=[3\ 4\ 5]$. The minimization of the indicator $\varepsilon_{1,N}$ was performed to obtain $N=5$ and to approximate $\delta$ with an absolute error of $2.24\times 10^{-12}$. See Figure \ref{deltaNsbf_ex5} for the behavior of the indicator in the first and last iterations.  

\begin{figure}[H]
        \centering
        \includegraphics[scale=0.55]{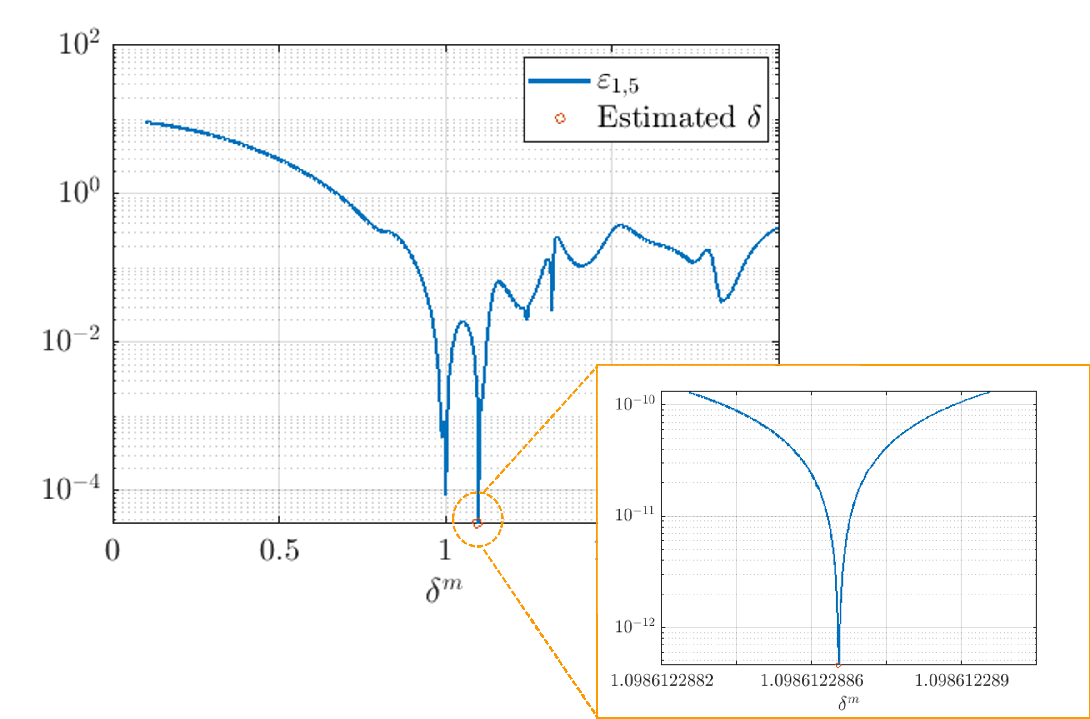} 
    \caption{ Approximation of $\delta$ of Example \ref{exam5} by using NSBF series, from the first and the last iterations.}
    \label{deltaNsbf_ex5}
\end{figure}
\end{example}

\begin{example} \label{exam6}
We now study the recovery of $\delta$ corresponding to the refractive index of Example~\ref{exam3}. The unknown $\delta$ is first estimated using two approaches based on the density formula~(\ref{aformula2}), as illustrated in Figure~\ref{delta_ex6}. Both methods utilize all the $204$ eigenvalues within the strip $0 < \operatorname{Re} k < \mathcal{R} = 300$ (shown in Figure \ref{ex3}).

The first approach directly applies the density formula to the largest $\mathcal{N}$-value in the dataset. Specifically, $\delta$ is approximated using $ \tilde{\delta} =\mathcal{N} \pi/\mathcal{R} - 1$, where $\mathcal{R}$ is the corresponding magnitude of the eigenvalue. This calculation yields an estimated $ \tilde{\delta}$, with an absolute error of $6.06 \times 10^{-3}$ to the exact value (Figure~\ref{delta_ex6}, left). In the second approach, a grid search is performed over the interval $[0.1, 2]$ to minimize the mean absolute error between the approximate $\tilde{\delta}$ values and candidate $\delta_{grid}$ values. Through this minimization, the estimated $\delta_{grid}$ has an absolute error of $2.04 \times 10^{-2}$ to the exact value (Figure~\ref{delta_ex6}, right).

\begin{figure}[H]
\begin{centering}
    \begin{minipage}[c][1\totalheight][t]{0.48\textwidth}
        \centering
        \includegraphics[scale=0.55]{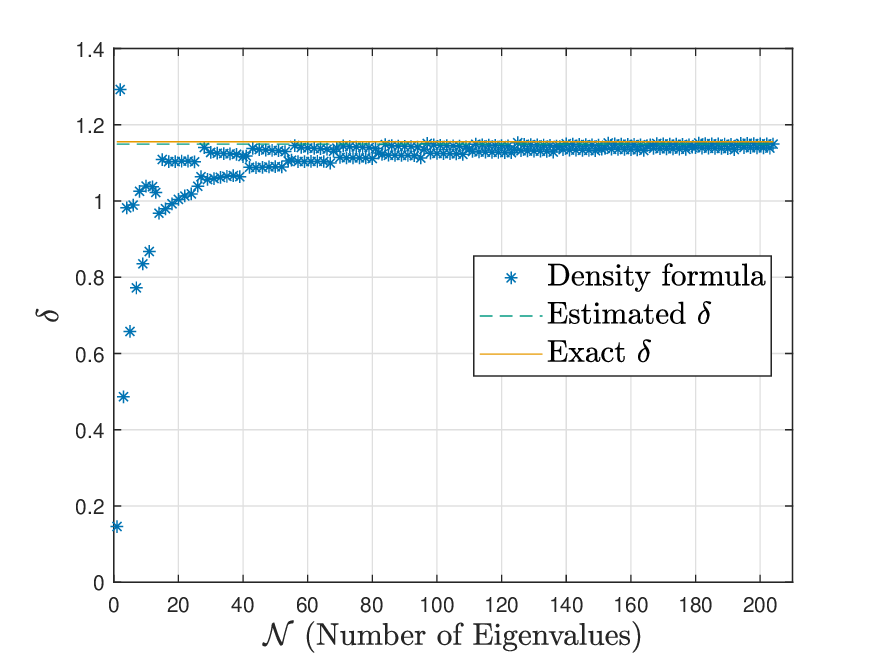}
    \end{minipage}
    \begin{minipage}[c][1\totalheight][t]{0.48\textwidth}
        \centering
        \includegraphics[scale=0.55]{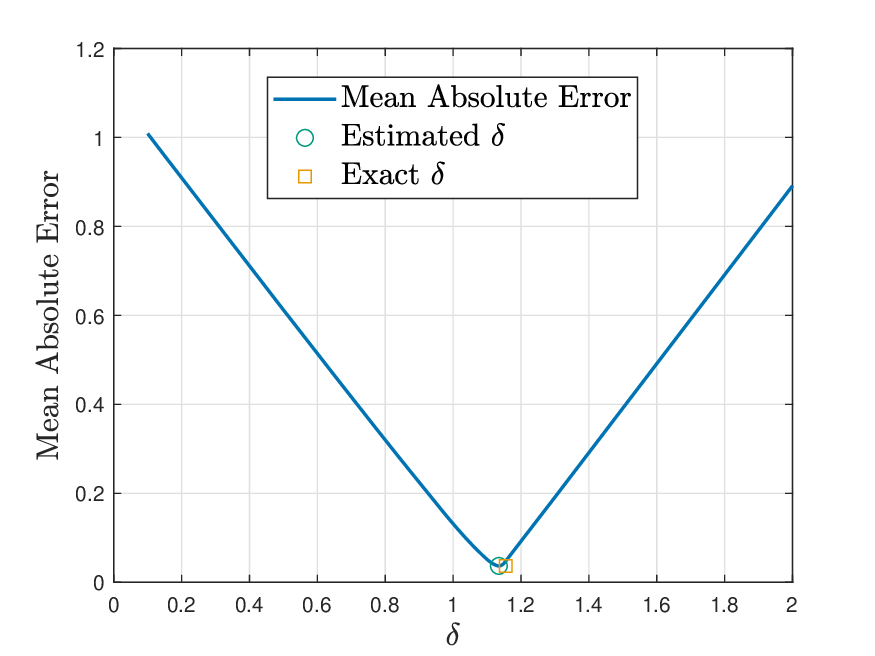}
    \end{minipage}
     \caption{ Approximation density formula for $\delta$ (left) and error minimization (right) of Example \ref{exam6}.}
    \label{delta_ex6}
\end{centering}
\end{figure}

Additionally, we estimate $\delta$ using Algorithm \ref{algo_delta}. In the first test, the input data consist of the 40 transmission eigenvalues with the lowest magnitudes (a considerably smaller set), and the grid $N=[8 \ 10 \ 15 \ 18 \ 20]$. In the second test, we use 150 transmission eigenvalues and the grid $N=[8 \ 10 \ 14 \ 18 \ 22 \ 26]$. Using 40 eigenvalues results in $N=15$  and an absolute error of $6.05\times 10^{-7}$ for $\delta$. With 150 input eigenvalues, we obtain $N = 18$ and an absolute error of $2.83 \times 10^{-11}$ for $\delta$. Figure \ref{deltaNsbf_ex6} shows the indicators in the final iteration for each test.  

\begin{figure}[H]
\begin{centering}
    \begin{minipage}[c][1\totalheight][t]{0.48\textwidth}
        \centering
        \includegraphics[scale=0.55]{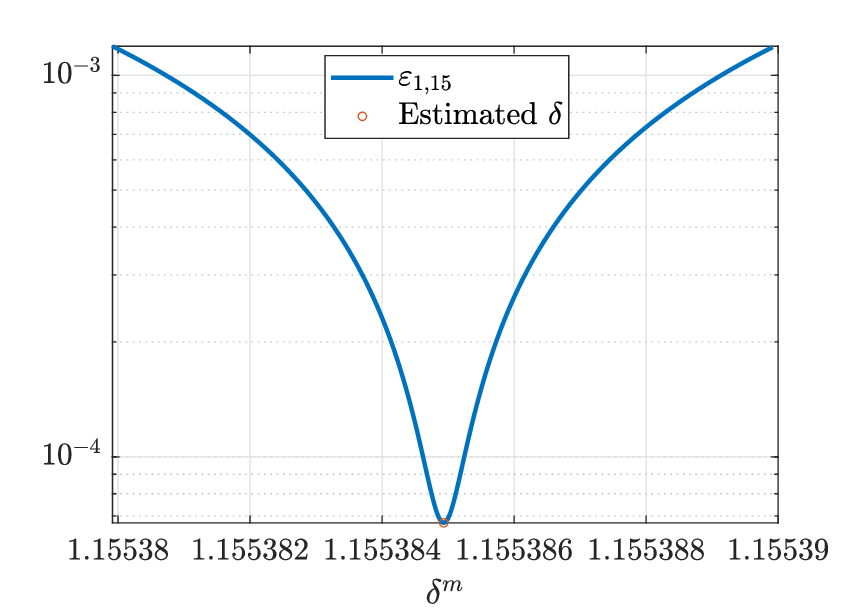}
    \end{minipage}
    \begin{minipage}[c][1\totalheight][t]{0.48\textwidth}
        \centering
        \includegraphics[scale=0.55]{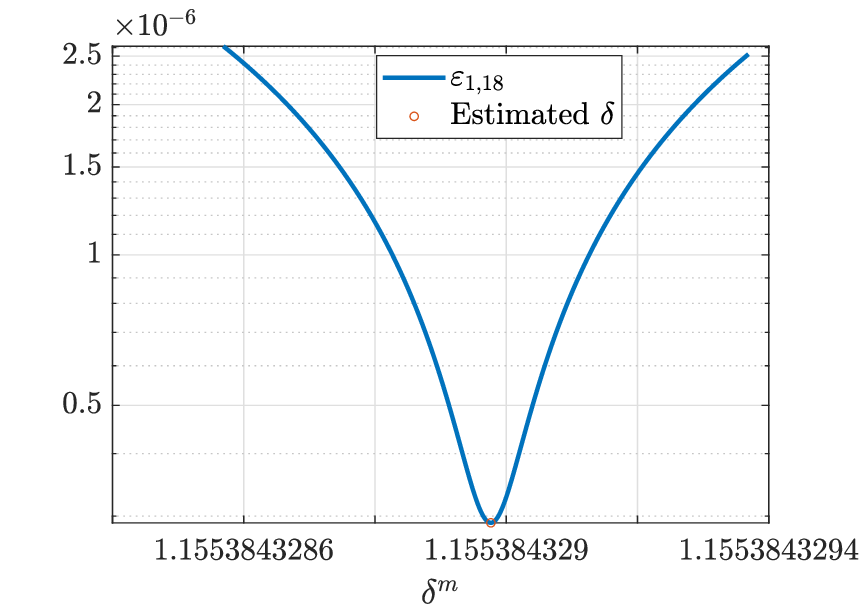}
    \end{minipage}
     \caption{ Approximation of $\delta$ from NSBF representations by using 40 eigenvalues (left) and by using 150 eigenvalues (right) of Example \ref{exam6}.}
   \label{deltaNsbf_ex6}
\end{centering}
\end{figure}

\end{example}

\begin{example} \label{exam7}
We recover $\delta=1$ corresponding to the refractive index of Example~\ref{ExampleDelta1} using Algorithm \ref{algo_delta}. Note that neither formula \eqref{aformula} nor \eqref{aformula2} are applicable in this case. Here, the input data of Algorithm \ref{algo_delta} consist of the $8$ lowest magnitude transmission eigenvalues and $N=[2 \ 3 \ 4]$. The output of the algorithm is $N=3$ and $\delta$ computed with an absolute error of $1.91\times10^{-4}$, presented in Figure \ref{deltaNsbf_ex7}.

\begin{figure}[h!]
        \centering
        \includegraphics[scale=0.55]{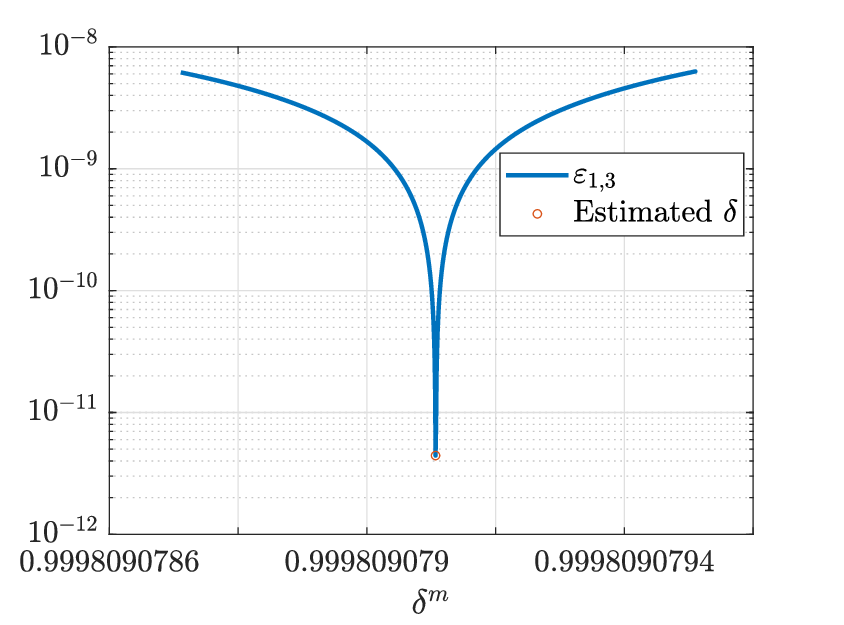}
    \caption{ Approximation of $\delta$ of Example \ref{exam7} by using NSBF series.}
    \label{deltaNsbf_ex7}
\end{figure}
\end{example}

From the above results, we notice that $\delta$ is recovered with high accuracy even when only a few eigenvalues are available, in contrast to the asymptotic formulas, which require more data and are not always applicable. Furthermore, in the subsequent section where we consider the reconstruction of refractive indices from transmission eigenvalues, we explore the approximation of $\delta$ in more detail as the number of input eigenvalues varies.

\subsubsection{Numerical solution of the inverse problem}  \label{sec_exampl_inv_index}

In this section, we apply Algorithm \ref{algo_inv} to solve the inverse TEP and demonstrate its efficiency in accurately recovering the refractive index. We evaluate the algorithm's performance by gradually increasing the number of eigenvalues employed. Moreover, the unknown parameter $\delta$ is estimated for each example using Algorithm \ref{algo_delta}, which also allows us to observe how the recovery of $\delta$ evolves with the increasing number of eigenvalues. 

We note that treating  $\delta$ as an additional unknown substantially increases the complexity of our inverse problem algorithm. Since $\delta$ defines the right endpoint of the Sturm–Liouville interval in \eqref{eq:SL} and enters every key relation—such as the system \eqref{NewSystemReduced}—even small errors in its approximation can destabilize the refractive index reconstruction. Nonetheless, as we demonstrate below, our approach successfully mitigates these difficulties.

\begin{example} \label{exam8} 
By using the lowest $10$ complex eigenvalues computed from the NSBF representations
in Example \ref{exam1}, the refractive index  $n(r)=16/\left((r+1)(3-r)\right)^{2}$ is recovered using three different values of $\delta$:  two previously determined in Example \ref{exam5}, and the exact value. See Figure \ref{inv1} for the reconstructions and absolute errors.  As observed, solving the inverse problem is particularly sensitive to the value of $\delta$.

\begin{figure}[H]
\begin{minipage}{0.49\textwidth}
        \centering
        \includegraphics[scale=0.6]{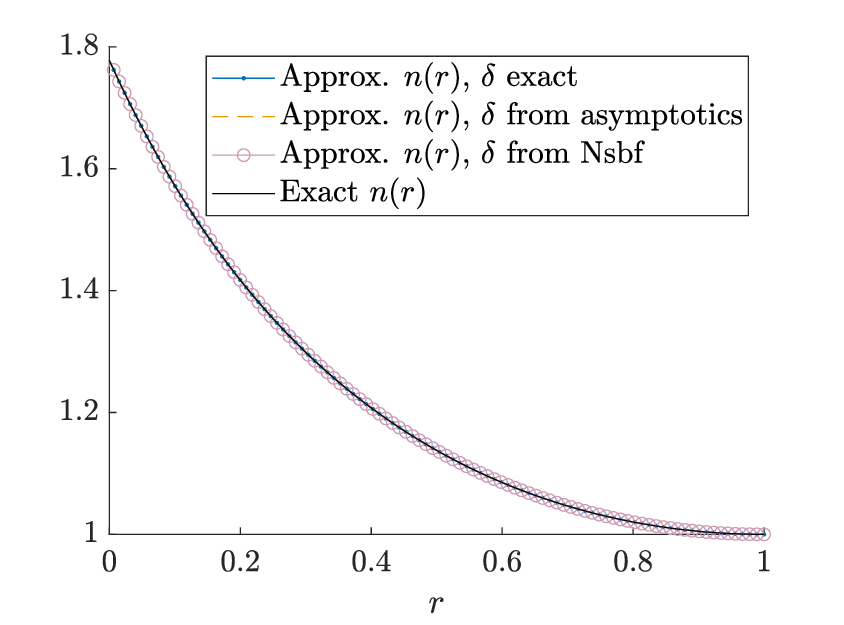}
    \end{minipage}
    \hfill
        \begin{minipage}{0.49\textwidth}
        \centering
        \includegraphics[scale=0.6]{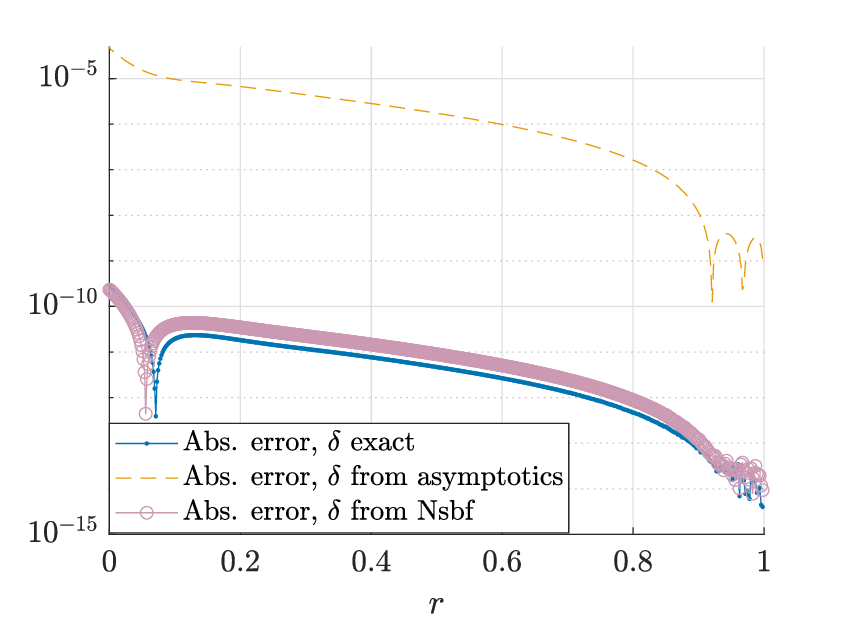}
         \end{minipage}
        \caption{ Recovered refractive index $n(r)=16/\left((r+1)(3-r)\right)^{2}$ from $10$ eigenvalues (left) and absolute error of the reconstruction (right) for different values of $\delta$ of Example \ref{exam8}.}
    \label{inv1}
\end{figure}

Additionally, we consider different cases for the number of given eigenvalues.  Using the corresponding $\delta$ approximations shown in Table \ref{table:1}, we then recover the refractive index, as presented in Figure \ref{invExample1}.

\begin{table}[H]
\centering
\caption{Approximation of $\delta$ from an increasing number of eigenvalues for Example \ref{exam8}.}
\label{table:1}
\begin{tabular}{lcccc} 
  \toprule
  Number of eigs &6 & 10 & 30 & 50  \\
  \midrule
  Number of coefs $N$ & 3 & 5  & 6  & 9   \\
  Abs. Error $\delta$     & $7.82\times 10^{-8}$ & $2.24\times 10^{-12}$ & $1.58\times 10^{-13}$ & $9.59\times 10^{-13}$ \\
  \bottomrule
\end{tabular}
\end{table}

\begin{figure}[H]
\begin{minipage}{0.49\textwidth}
        \centering
        \includegraphics[scale=0.6]{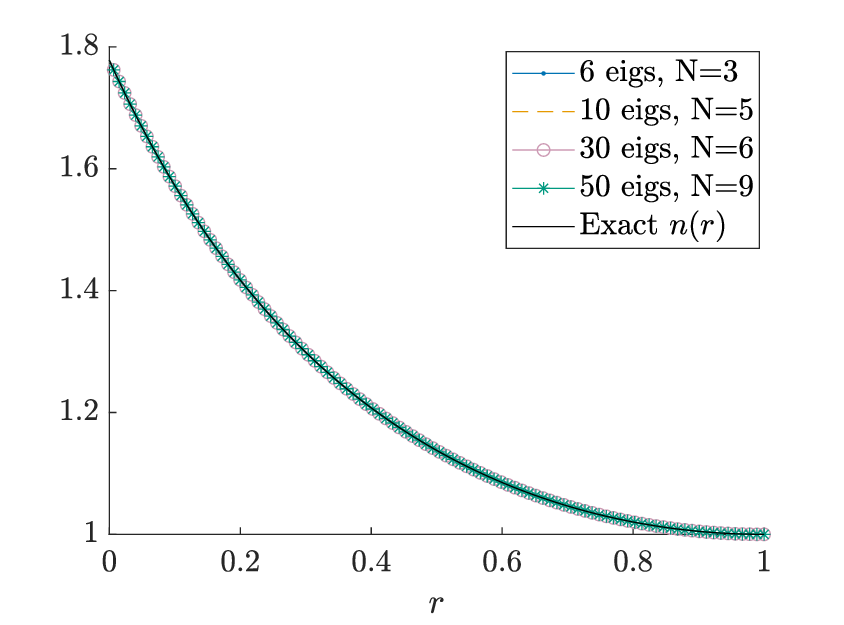}
    \end{minipage}
    \hfill
        \begin{minipage}{0.49\textwidth}
        \centering
        \includegraphics[scale=0.6]{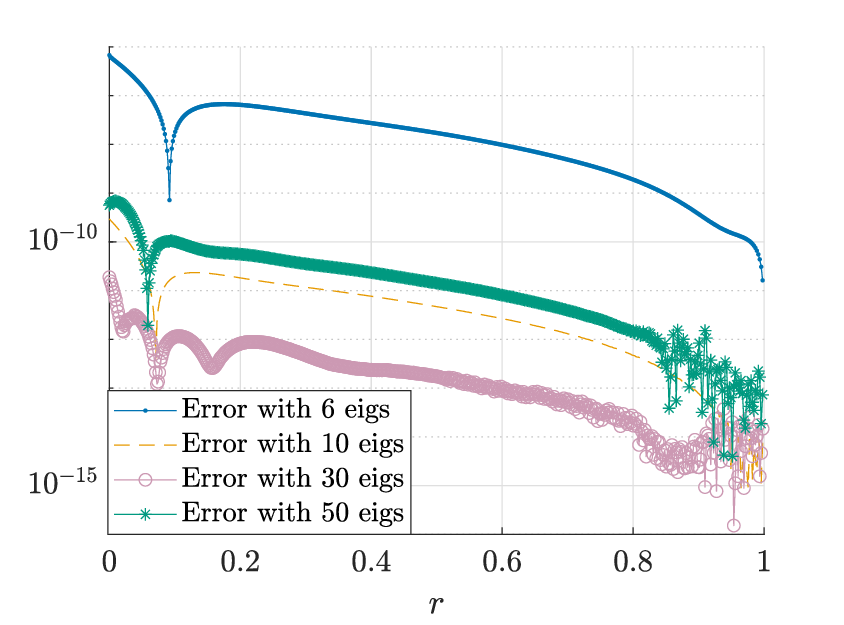}
         \end{minipage}
        \caption{ Recovered refractive index $n(r)=16/\left((r+1)(3-r)\right)^{2}$ (left) from $6$, $10$, $30$ and $50$ eigenvalues and absolute error of the reconstruction (right) of Example \ref{exam8}.}
    \label{invExample1}
\end{figure}

\end{example}

\begin{example} \label{exam9}
We study the inverse problems to recover the refractive indices considered in Example \ref{exam2}, by using both real and complex eigenvalues from the NSBF representations. 
For $n_1(r)=1/(1+(1-r)^{2})^{2}$, the approximation of $\delta$ using an increasing number of eigenvalues is given in Table \ref{table:2}. 

\begin{table}[H]
\centering
\caption{Approximation of $\delta$ from an increasing number of eigenvalues for $n_1(r)$ of Example ~\ref{exam9}.}
\label{table:2}
\begin{tabular}{lcccc} 
  \toprule
  Number of eigs      & 7 & 8 & 9 & 10 \\ 
  \midrule
  Number of coefs $N$  & 2 & 3 & 3 & 4 \\ 
  Abs. Error $\delta$ & $1.3\times10^{-4}$ & $3.25\times10^{-7}$ & $3.48\times10^{-7}$ & $5.23\times10^{-9}$ \\ 
  \bottomrule
\end{tabular}
\end{table}

In Figure ~\ref{inv_n1}, we present the reconstructions for the refractive index $n_1(r)$ and the corresponding absolute errors. 

\begin{figure}[H]
    \begin{minipage}{0.49\textwidth}
        \centering
        \includegraphics[scale=0.6]{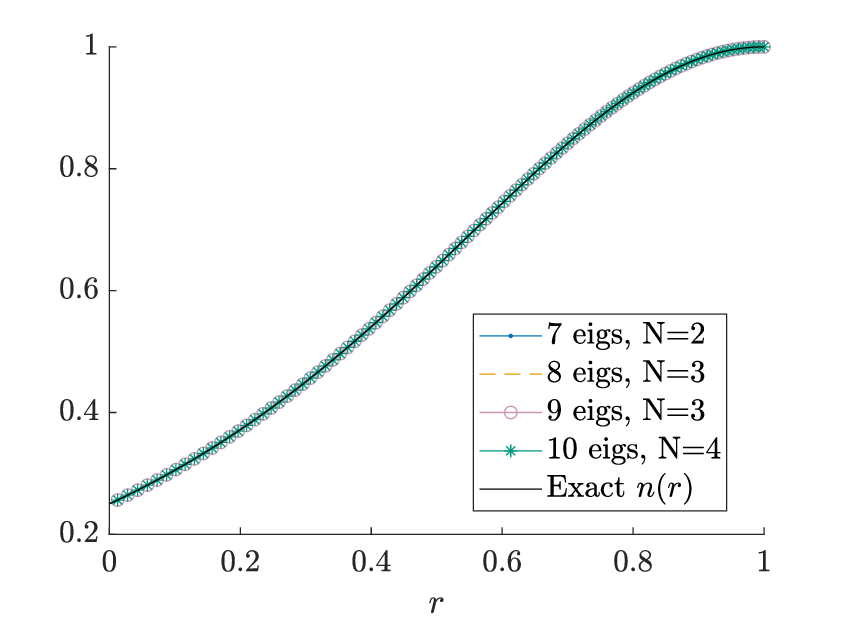}
    \end{minipage}
    \hfill
    \begin{minipage}{0.49\textwidth}
        \centering
        \includegraphics[scale=0.6]{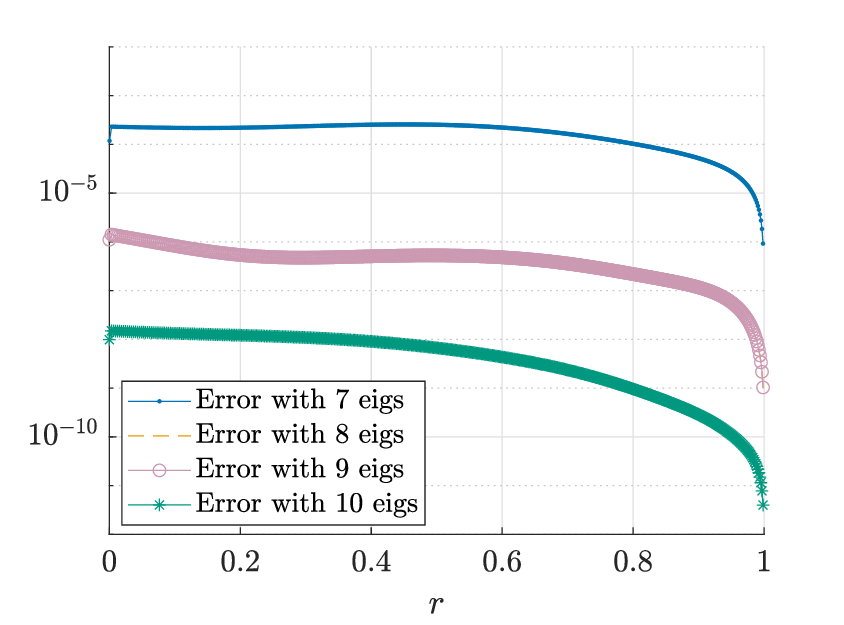}
    \end{minipage}
    \caption{  Recovered refractive index $n_1(r)=1/(1+(1-r)^{2})^{2}$ from $7$, $8$, $9$ and $10$ eigenvalues (left) and absolute error of the reconstruction (right) of Example \ref{exam9}  (the error curves for $8$ and $9$ eigenvalues are overlapping).}
    \label{inv_n1}
\end{figure}

For $n_2(r)=\left(\pi/4\right)^{2}$, the recovery of $\delta$ is presented in Table  \ref{table:3}. Furthermore, 
the reconstructions of the refractive index and the corresponding absolute errors are shown in Figure ~\ref{inv_n2}. 

\begin{table}[H]
\centering
\caption{Approximation of $\delta$ from an increasing number of eigenvalues for $n_2(r)$ of Example \ref{exam9}.}
\label{table:3}
\begin{tabular}{lccccc}
  \toprule
  Number of eigs & 6 & 7 & 8 & 9 & 10 \\
  \midrule
  Number of coefs $N$ & 2 & 2 & 3 & 4 & 4 \\
  Abs. Error $\delta$ & 0.53 & $2.44\times10^{-15}$ & $2.55\times10^{-15}$ & $2.44\times10^{-15}$ & $2.66\times10^{-15}$ \\
  \bottomrule
\end{tabular}
\end{table}

\begin{figure}[H]
    \begin{minipage}{0.49\textwidth}
        \centering
        \includegraphics[scale=0.6]{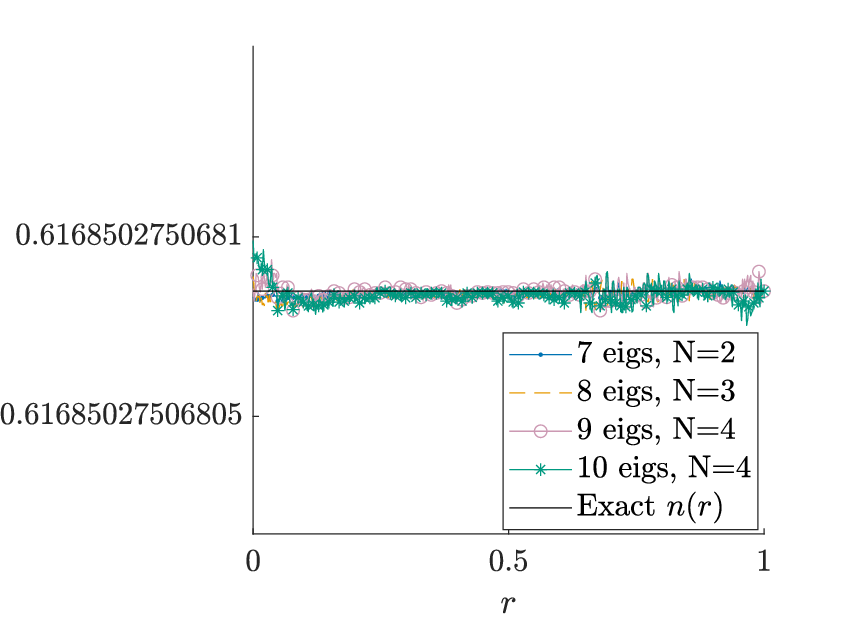}
    \end{minipage}
    \hfill
    \begin{minipage}{0.49\textwidth}
        \centering
        \includegraphics[scale=0.6]{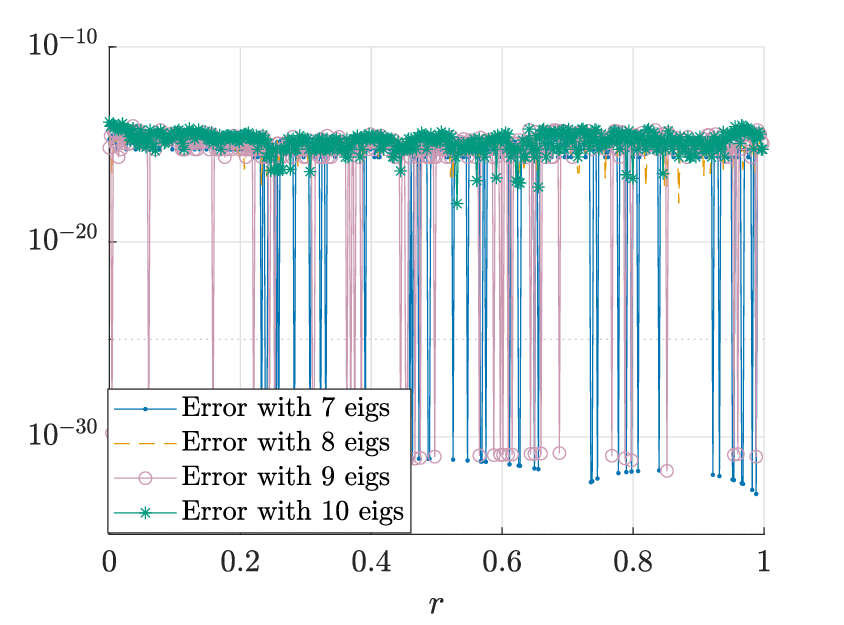}
    \end{minipage}
    \caption{  Recovered refractive index $n_2(r)=\left(\pi/4\right)^{2}$ from $7$, $8$, $9$ and $10$ eigenvalues (left) and absolute error of the reconstruction (right) of Example \ref{exam9}.}
    \label{inv_n2}
\end{figure}

Finally, for $n_3(r)=\left(1+0.4292\left(r-1\right)\right)^{2}$, we refer to Table \ref{table:4} for the approximation of  $\delta$, and to Figure \ref{inv_n3} for the reconstructions of the refractive index and the corresponding absolute errors. 

\begin{table}[H]
\centering
\caption{Approximation of $\delta$ from an increasing number of eigenvalues for $n_3(r)$ of Example \ref{exam9}.}
\label{table:4}
\begin{tabular}{lccccc}
 \toprule
Number of eigs & 5 & 6 & 7 & 8 & 9 \\
\midrule
Number of coefs $N$ & 1 & 2 & 2 & 4 & 4 \\
Abs. Error $\delta$ & $4.7\times10^{-3}$ & $3.6\times10^{-4}$ & $3.39\times10^{-4}$ & $1.66\times10^{-6}$ & $7.99\times10^{-5}$ \\
\bottomrule
\end{tabular}
\end{table}

\begin{figure}[H]
    \begin{minipage}{0.49\textwidth}
        \centering
        \includegraphics[scale=0.6]{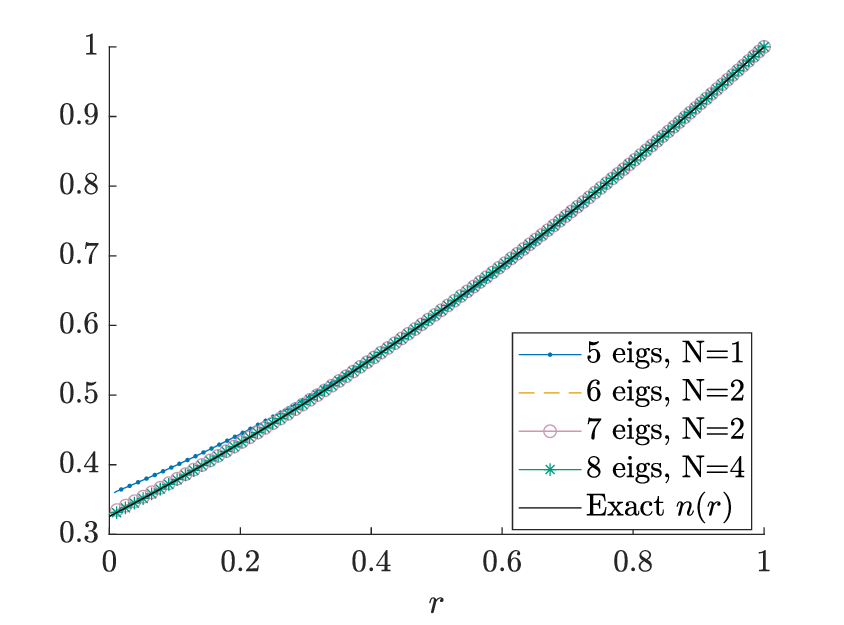}
    \end{minipage}
    \hfill
    \begin{minipage}{0.49\textwidth}
        \centering
        \includegraphics[scale=0.6]{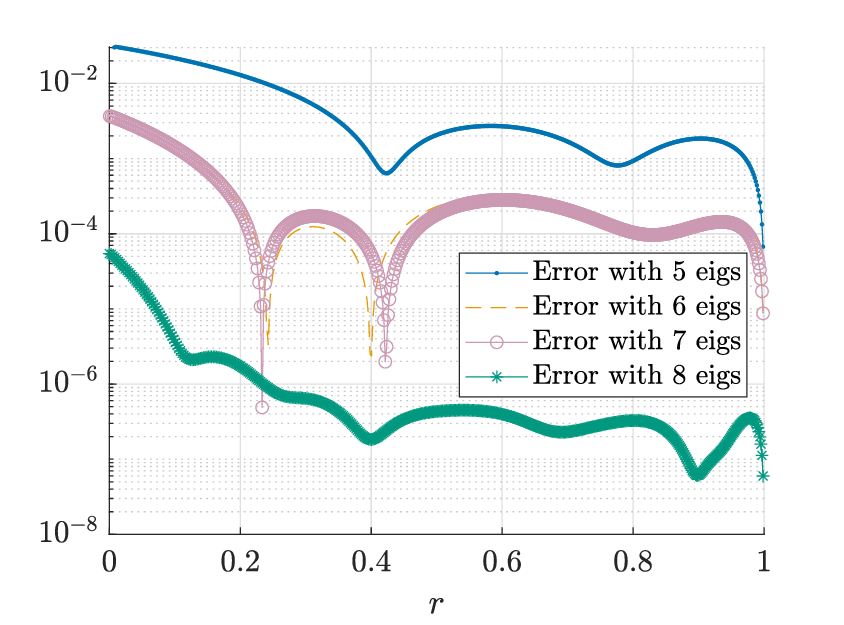}
    \end{minipage}
    \caption{  Recovered refractive index $n_3(r)=\left(1+0.4292\left(r-1\right)\right)^{2}$ from $5$, $6$, $7$ and $8$ eigenvalues (left) and absolute error of the reconstruction (right) of Example \ref{exam9}.}
    \label{inv_n3}
\end{figure}
\end{example}

\begin{example} \label{exam10}
We now study the reconstruction of the refractive index  $n(r)=1.2 + (1 - r)\sin(2\pi r)$. The eigenvalues are calculated from the NSBF representations in Example \ref{exam3}. As in the previous examples, we consider an increasing number of input eigenvalues, from which we first approximate $\delta$ as presented in Table \ref{table:5}. Then, the corresponding reconstructions and their absolute errors are given in Figure \ref{invexam10}.

\begin{table}[H]
\centering
\caption{Approximation of $\delta$ from an increasing number of eigenvalues for Example \ref{exam10}.}
\label{table:5}
\begin{tabular}{lccccc} 
\toprule
Number of eigs & 30 & 40 & 50 & 100 & 150 \\
\midrule
Number of coefs $N$   & 14 & 15 & 16 & 13  & 18 \\
Abs. Error $\delta$    & $4.07\times10^{-6}$ & $6.05\times10^{-7}$ & $4.49\times10^{-8}$ & $9.53\times10^{-10}$ & $2.83\times10^{-11}$ \\
\bottomrule
\end{tabular}
\end{table}

\begin{figure}[H]
\begin{minipage}{0.49\textwidth}
        \centering
        \includegraphics[scale=0.6]{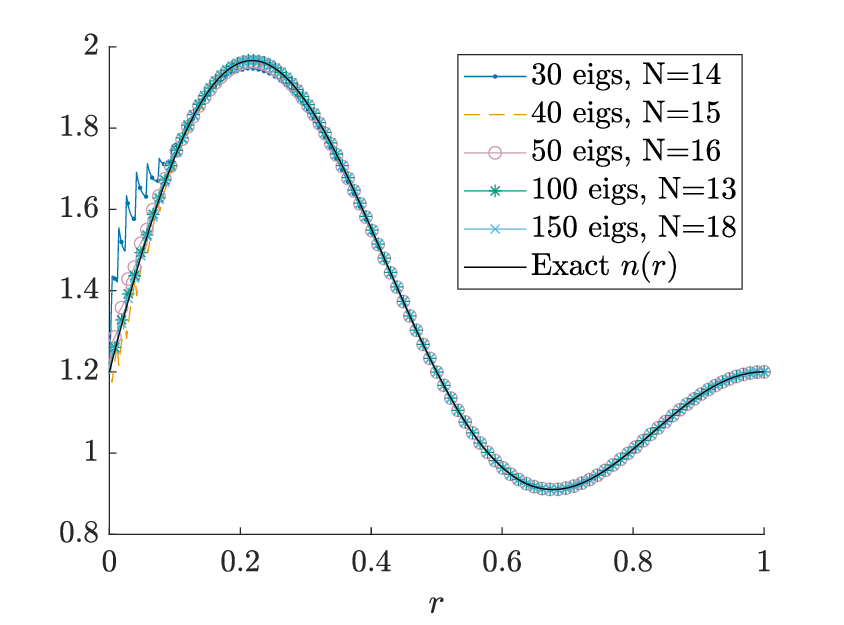}
    \end{minipage}
    \hfill
        \begin{minipage}{0.49\textwidth}
        \centering
        \includegraphics[scale=0.6]{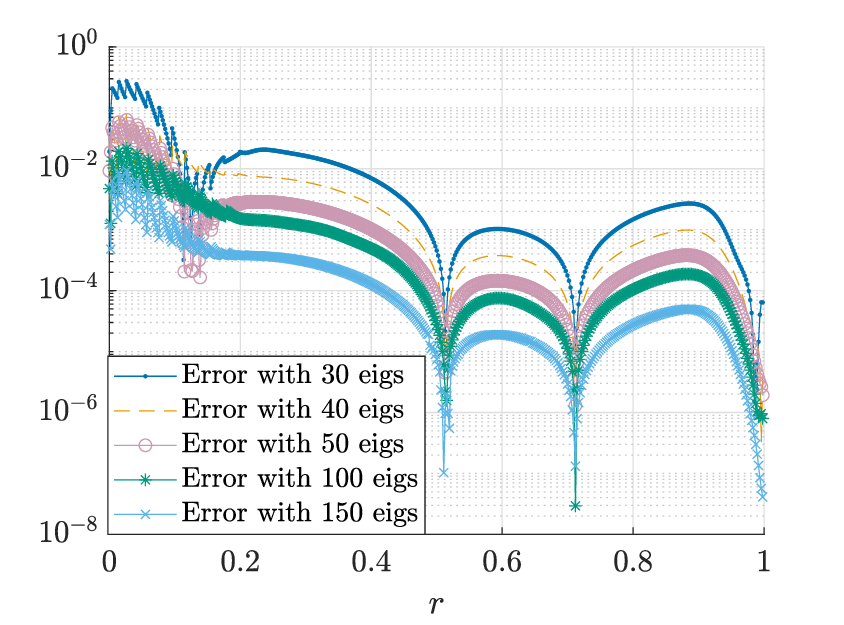}
         \end{minipage}
        \caption{ Recovered refractive index $n(r)=1.2 + (1 - r)\sin(2\pi r)$ from $30,\ 40,\ 50,\ 100$ and $150$ eigenvalues (left) and absolute error of the reconstruction (right) of Example \ref{exam10}.}
    \label{invexam10}
\end{figure}
\end{example}

\begin{example} \label{exam11} In this example, we consider the inverse problem for the refractive index  $n(r)=(r+0.5)^2$, where the corresponding eigenvalues are obtained in Example \ref{ExampleDelta1}. Only real eigenvalues are used in this case. The approximations of $\delta$ are shown in Table \ref{table:6} and the respective reconstructions of the refractive index in Figure \ref{invexam11}. 

\begin{table}[H]
\centering
\caption{Approximation of $\delta$ from an increasing number of eigenvalues for Example \ref{exam11}.}
\label{table:6}
\begin{tabular}{lccccc}
\toprule
Number of eigs & 8 & 9 & 10 & 11 & 12 \\
\midrule
Number of coefs N   & 3 & 4 & 4  & 5  & 4  \\
Abs. Error $\delta$ & $1.91\times10^{-4}$ & $5.46\times10^{-5}$ & $3.75\times10^{-5}$ & $8.32\times10^{-5}$ & $1.2\times10^{-5}$ \\
\bottomrule
\end{tabular}
\end{table}

\begin{figure}[H]
\begin{minipage}{0.49\textwidth}
        \centering
        \includegraphics[scale=0.6]{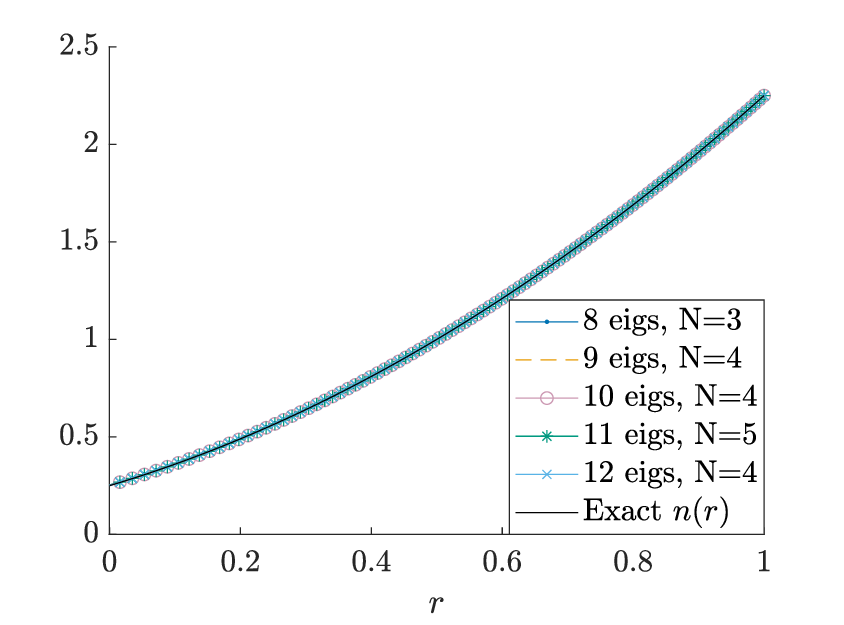}
    \end{minipage}
    \hfill
        \begin{minipage}{0.49\textwidth}
        \centering
        \includegraphics[scale=0.6]{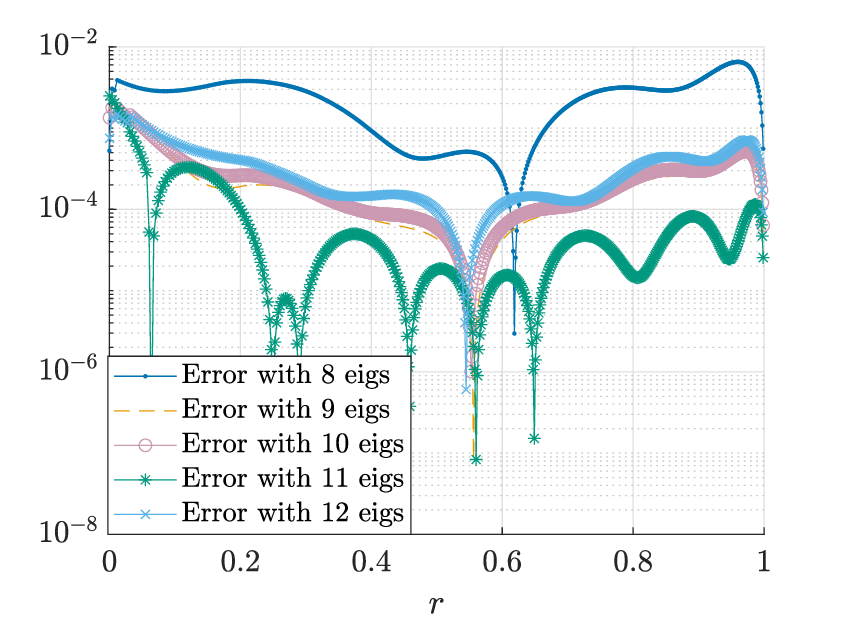}
         \end{minipage}
        \caption{ Recovered refractive index $n(r)=(r+0.5)^2$ from $8,\ 9,\ 10,\ 11$ and $12$ eigenvalues (left) and absolute error of the reconstruction (right) of Example \ref{exam11}.}
    \label{invexam11}
\end{figure}
\end{example}

We note that in all the inverse problem examples studied above, no a priori assumptions on the value of $\delta$ or the sign of $1-n(r)$ are imposed. The only input data used are the transmission eigenvalues, along with the values of $n(1)$ and $n^{\prime}(1)$. As found in the literature, uniqueness theorems for the inverse problem typically require prior knowledge on the sign of $1-\delta$ or $1-n(r)$; see \cite[Section 4]{Pallreview} and the references therein for more details. 

\subsubsection{Application of spectrum completion and the inverse problem} \label{sec_exampl_inv_scompl}

We now present numerical examples to examine spectrum completion and the associated inverse problems. We consider various input-eigenvalue scenarios: real only, non-real only, and both real and complex. Focusing on real eigenvalues is of particular interest for the inverse TEP, since sampling methods based on far-field scattering data can detect only real eigenvalues \cite{CCHdeter}. Thus, using real eigenvalue inputs and completing the complex spectrum accordingly can be useful for applications where only real measurements are available.

\begin{example} \label{exam12}
We consider the spectrum completion corresponding to the refractive index of Example \ref{exam1}. As input, we use the five complex eigenvalues of smallest magnitude computed in Example \ref{exam1}. Using these five eigenvalues, we approximate $\delta$ with an absolute error of $1.25\times 10^{-5}$. Note that from these eigenvalues, we recover both the real and the following non-real eigenvalues, see Figure \ref{completion_ex12}.

To confirm the accuracy of the spectrum completion, we calculate the error with respect to the eigenvalues obtained from the closed-form characteristic equation, as described in Example \ref{exam1}. The maximum absolute error is $4.05\times 10^{-2}$.

\begin{figure}[H]
        \centering
        \includegraphics[scale=0.6]{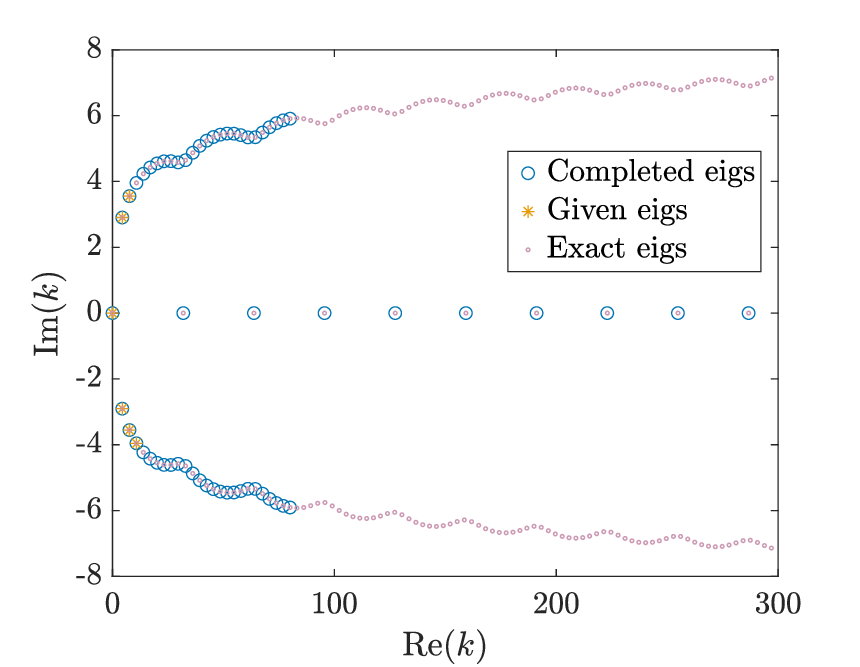}
        \caption{Spectrum completion from 5 complex eigenvalues of Example \ref{exam12}.}
    \label{completion_ex12}
\end{figure}

For this example, using the lowest 5 real eigenvalues, the absolute error of $\delta$ recovered by Algorithm \ref{algo_delta} resulted in $9.34\times 10^{-1}$.
This accuracy was not sufficient for an accurate enough spectrum completion. However, considering the same set of eigenvalues with the exact value of $\delta$, Algorithm \ref{algo_sc} was applied resulting in a reliable completion as shown in Figure \ref{CompletionEx1} (left). The maximum absolute error of this spectrum completion, with respect to the eigenvalues of the closed-form characteristic function is $9.35 \times 10^{-3}$. 

Furthermore, the refractive index was accurately recovered by using the $5$ real eigenvalues plus $11$ more completed eigenvalues. The maximum absolute error with the completed spectrum decreased compared to that obtained using only the $5$ real eigenvalues,  Figure \ref{CompletionEx1} (right). 

\begin{figure}[H]
\begin{minipage}{0.49\textwidth}
        \centering
        \includegraphics[scale=0.6]{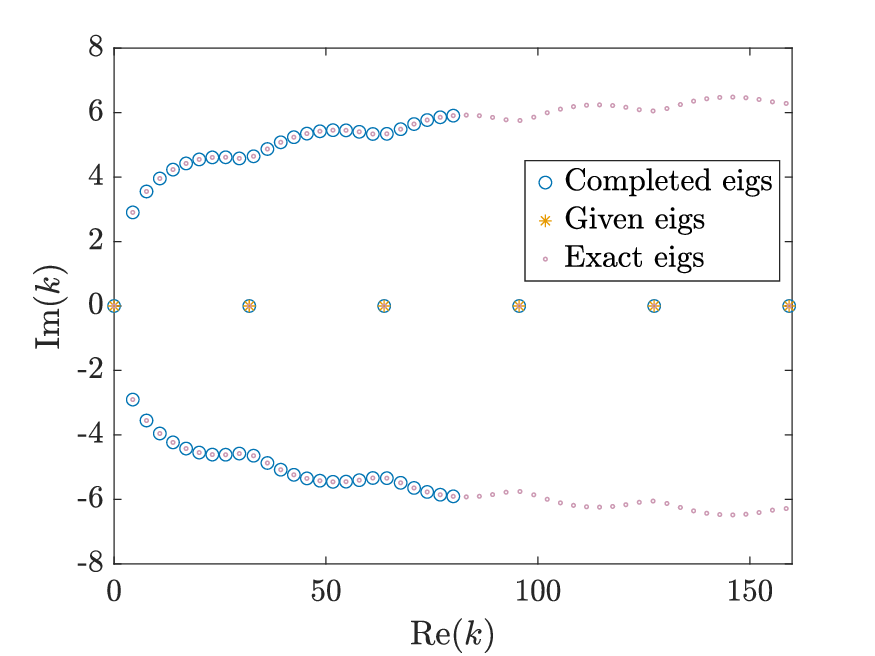}
    \end{minipage}
    \hfill
        \begin{minipage}{0.49\textwidth}
        \centering
        \includegraphics[scale=0.6]{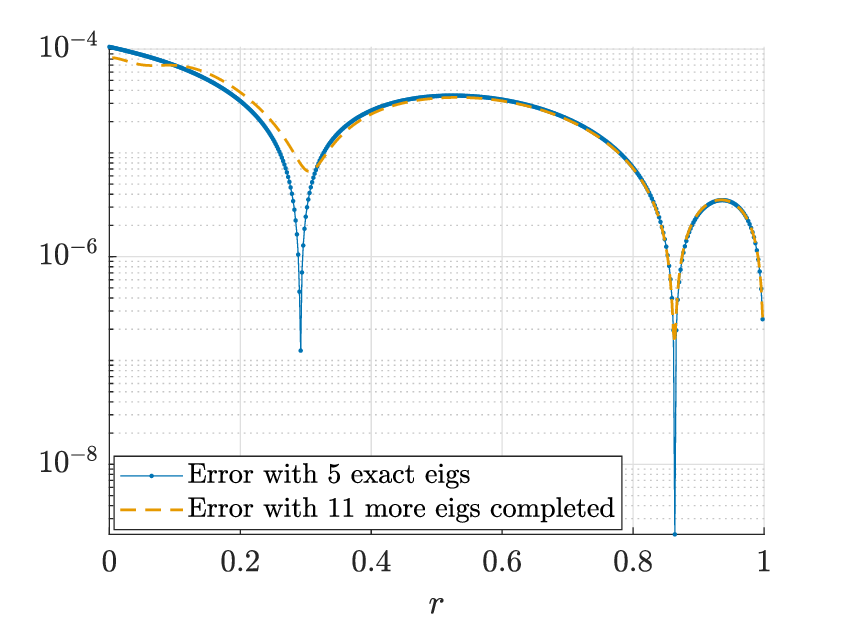}
         \end{minipage}
        \caption{ Spectrum completion from 5 real eigenvalues and exact $\delta$  (left), and absolute error of the reconstruction of $n(r)=16/\left((r+1)(3-r)\right)^{2}$ from $5$ real eigenvalues plus $11$  completed eigenvalues (right) of Example \ref{exam12}.}
    \label{CompletionEx1}
\end{figure}

Note that using more input eigenvalues yields a more accurate spectrum completion, allowing us to complete complex eigenvalues with even higher magnitudes than those presented in Figures \ref{completion_ex12} and \ref{CompletionEx1}. 

\end{example}
 \begin{example} \label{exam13}
We now study the spectrum completion corresponding to the refractive index of Example \ref{exam3} and its application for solving the inverse problem. The input data are the first 30 nonzero eigenvalues of the set computed in Example \ref{exam3}, and $\delta$ approximated as presented in the first column of Table \ref{table:5}. With these data, $140$ more transmission eigenvalues are located as presented in Figure \ref{deltaNsbf_ex13}.
\begin{figure}[H]
        \centering
        \includegraphics[scale=0.55]{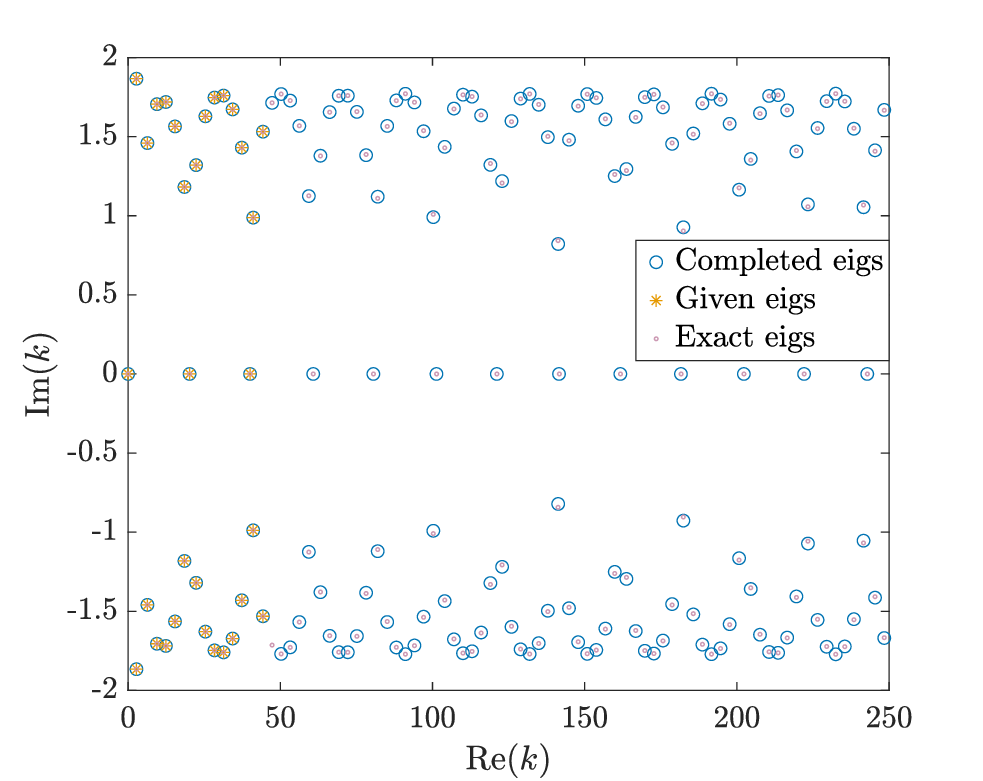}
    \caption{Spectrum completion from 30 eigenvalues of Example \ref{exam13}.}
    \label{deltaNsbf_ex13}
\end{figure}
Moreover, the completed spectrum is used to solve the inverse problem. The corresponding reconstructions and their absolute errors are shown in Figure \ref{invexam13}. The refractive index is recovered by using the $30$ given plus $70$ more completed eigenvalues. We observe that using the completed spectrum yields a more accurate reconstruction.

\begin{figure}[H]
\begin{minipage}{0.49\textwidth}
        \centering
        \includegraphics[scale=0.6]{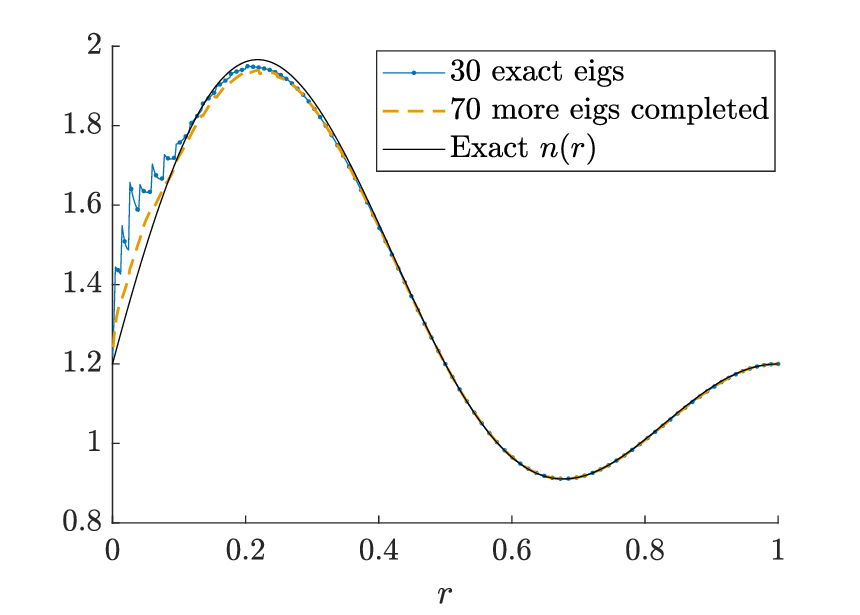}
    \end{minipage}
    \hfill
        \begin{minipage}{0.49\textwidth}
        \centering
        \includegraphics[scale=0.6]{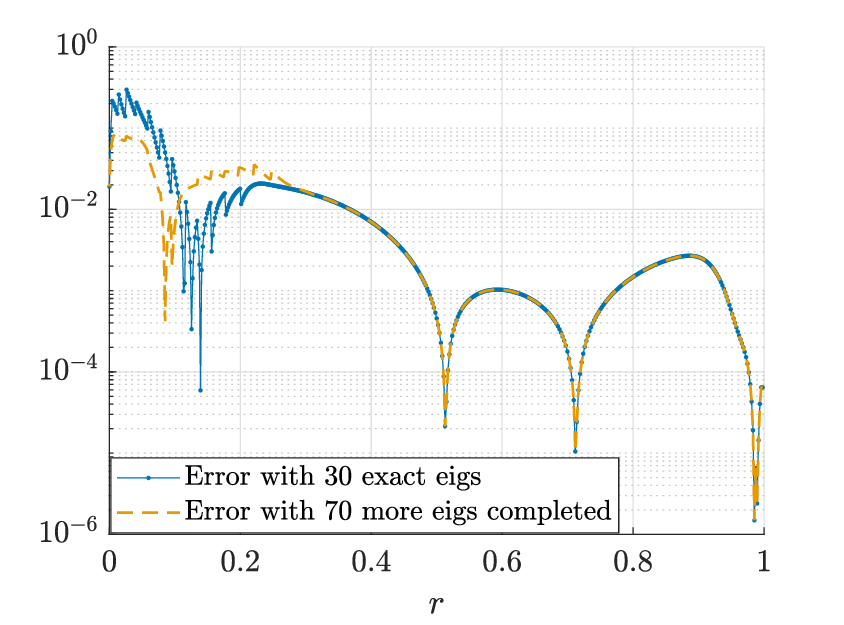}
         \end{minipage}
        \caption{ Recovered refractive index $n(r)=1.2 + (1 - r)\sin(2\pi r)$ from $30$ eigenvalues plus $70$  completed eigenvalues  (left) and absolute error of the reconstruction (right) of Example \ref{exam13}.}
    \label{invexam13}
\end{figure}
\end{example}

\section{Discussion and summary} \label{sec_fin}

In this paper, we have introduced a novel NSBF-based approach for both the direct and inverse TEP in the spherically symmetric setting,  addressing the absence of dedicated numerical methods for both real and complex spectrum in the existing TEP literature. 

By expanding the characteristic function of the transformed Sturm–Liouville problem in NSBF, the direct TEP reduces to the computation of a small set of NSBF coefficients, followed by the root‑finding of a truncated NSBF partial sum (Algorithm~\ref{algo_dir}). This methodology provides accurate results for both real and complex eigenvalues, as illustrated in Section \ref{sec_exampl_dir}.  

The inverse problem was formulated as a two‑stage procedure.  First, we recovered the unknown interval length $\delta$ directly from transmission eigenvalues by introducing a new NSBF methodology (Algorithm~\ref{algo_delta}).  Second, with $\delta$ determined, the refractive index  $n(r)$ is reconstructed by solving a linear system for the first NSBF coefficients (Algorithm~\ref{algo_inv}).  Numerical experiments in Section \ref{sec_exampl_inv} demonstrate that $\delta$ can be estimated with very high accuracy using a few eigenvalues. The resulting reconstructions of $n(r)$ exhibit low errors across all studied cases, including constant, variable, monotonic, or even oscillatory refractive indices, with no a priori assumptions on the sign of $1-n(r)$ or on the value of $\delta$. We observed that increasing the number of input eigenvalues (beyond a minimum threshold) does not always lead to better reconstructions of the refractive index. This is a direct consequence of requiring a larger number of NSBF coefficients $N$, which leads to bigger and therefore more unstable systems to solve. Nonetheless, the method remains robust across all examples. Finally, while we obtained the optimal number of coefficients $N$ using (\ref{IndDirect}), alternative criteria such as those in Remark~\ref{RemarkEpsilon} could similarly be applied.

Furthermore, we have extended our NSBF framework to spectrum completion for the TEP (Algorithm \ref{algo_sc}), enabling the recovery of a larger spectrum from only a few real and/or complex input eigenvalues. Our numerical examples demonstrate that we can complete the spectrum accurately, and when these spectra are feed into the inverse problem algorithms, the reconstructed refractive indices may exhibit improved accuracy. This spectrum completion thus offers a practical tool for applications constrained to limited measurements, broadening the applicability of our direct and inverse TEP methodology.

Moreover, additional numerical experiments covering a wider range of refractive indices, alongside with algorithmic refinements and optimization strategies, would be valuable to further validate and enhance the performance of our NSBF-based approach. A rigorous convergence and stability analysis, while of interest, is beyond the scope of the present work and will be pursued in subsequent studies.
Future work could explore the extension of the proposed NSBF methodology to other classes of spherically symmetric direct and inverse transmission eigenvalue problems. These include the problem without the assumption of axially symmetric eigenfunctions (that is, for angular numbers $l\geq 1$) \cite{CCG, XYX}, the discontinuous TEP \cite{GP}, the problem involving a complex-valued refractive index that corresponds to absorbing medium \cite{CCHabs}, the anisotropic TEP with cavity \cite{KirschAsatryan}, and the recently introduced modified TEP \cite{GPS2}. These directions, among others, could further extend the applicability and impact of the proposed approach.\\

\textbf{CRediT authorship contribution statement} \\ 
\textbf{Vladislav V. Kravchenko:} Conceptualization, Formal analysis, Investigation, Methodology, Software, Supervision, Validation, Writing – original draft, Writing – review and editing, \textbf{L. Estefania Murcia-Lozano:} Conceptualization, Data curation, Formal analysis, Investigation, Methodology, Software, Validation, Visualization, Writing – original draft, Writing – review and editing, 
\textbf{Nikolaos Pallikarakis} Conceptualization, Data curation, Formal analysis, Investigation, Methodology,  Software, Validation, Visualization, Writing – original draft, Writing – review and editing. \\

\textbf{Declaration of competing interest} \\ 
The authors declare that they have no known competing financial interests or personal relationships that could have appeared to
influence the work reported in this paper. 
\\

\textbf{Funding} \\ 
Research of V. V. Kravchenko was supported by CONAHCYT, Mexico, via the grant "Ciencia de Frontera" FORDECYT - PRONACES/ 61517/ 2020. Research of L. Estefania Murcia-Lozano was supported by  Regional Mathematical Center of Southern Federal
University under the program of the Ministry of Education and Science of Russia, agreement No.
075-02-2025-1720.\\

\textbf{Data availability} \\ 
Data will be made available on request.

\bibliographystyle{plain}  

\bibliography{biblio}

\end{document}